\documentclass{article}
\usepackage[utf8]{inputenc}
\usepackage{amssymb,amsthm,amsmath}

\title{Coupled K\"ahler-Ricci solitons on toric Fano manifolds}
\author{Jakob Hultgren}
\date{}

\newtheorem{theorem}{Theorem}
\newtheorem{lemma}{Lemma}

\newtheorem{corollary}{Corollary}
\theoremstyle{definition}
\newtheorem{remark}{Remark}\theoremstyle{remark}
\theoremstyle{definition}
\newtheorem{conjecture}{Conjecture}

\newcommand{\Z}{\mathbb{Z}}
\newcommand{\R}{\mathbb{R}}
\newcommand{\C}{\mathbb{C}}

\newcommand{\M}{\mathcal{M}}

\renewcommand\Im{\operatorname{Im}}

\newcommand\Ric{\operatorname{Ric}}

\newcommand\Psh{\operatorname{Psh}}
\newcommand\Vol{\operatorname{Vol}}
\renewcommand\max{\operatorname{max}}
\renewcommand\div{\operatorname{div}}
\newcommand\Aut{\operatorname{Aut}}

\begin{document}
\maketitle

\begin{abstract}
We prove a necessary and sufficient condition in terms of the barycenters of a collection of polytopes for existence of coupled K\"ahler-Einstein metrics on toric Fano manifolds. This confirms the toric case of a coupled version of the Yau-Tian-Donaldson conjecture. We also obtain a necessary and sufficient condition for existence of torus-invariant solutions to a system of soliton type equations on toric Fano manifolds. Some of these solutions provide natural candidates for the large time limits of a certain geometric flow generalizing the K\"ahler-Ricci flow. 
\end{abstract}

\section{Introduction}
Given a compact K\"ahler manifold $(X,\omega)$, an important question in complex geometry is the problem of finding a metric of constant scalar curvature in the K\"ahler class $[\omega]$. It has been known for a long time that there are deep obstructions to existence of these metrics. In the case when $[\omega] = \pm c_1(X)$, constant scalar curvature metrics coincide with K\"ahler-Einstein metrics, i.e. metrics that are proportional to their Ricci tensor. It was recently showed \cite{ChenDonaldsonSun} that existence of such metrics is equivalent to a certain algebraic stability condition: K-polystability (see also \cite{Tian15}). A similar stabilitiy condition for general K\"ahler classes is conjectured to be equivalent to existence of constant scalar curvature metrics. However, except for in some special classes of manifolds (see \cite{Donaldson09}) this is open. It should also be pointed out that even in the light of \cite{ChenDonaldsonSun}, determining if a given manifold admits a K\"ahler-Einstein metric is not a straight forward task. The condition of K-polystability is not readily checkable. On the other hand, a large class of manifolds where K-polystability reduces to a simple criterion is given by toric Fano manifolds. Here, as was originally proved in \cite{WangZhu,ZhouZhu}, K-polystability and existence of K\"ahler-Einstein metrics is equivalent to the condition that the barycenter of the polytope associated to the anti-canonical polarization is the origin. In addition, \cite{WangZhu} proves that any toric Fano manifold admits a K\"ahler-Ricci soliton,
in other words a metric $\omega$ such that
\begin{equation}
    \Ric \omega = L_V(\omega)+\omega 
    \label{eq:ClassicalSoliton} 
\end{equation}
for a holomorphic vector field $V$. Here $L_V$ denotes Lie derivative along $V$. These appear as natural long time solutions to the K\"ahler-Ricci flow and have attracted great interest over the years. (see for example \cite{Hamilton93}, \cite{Hamilton95}, \cite{Cao} and \cite{Tian}). 

In a recent paper Witt Nystr\"om together with the present author introduced the concept of coupled K\"ahler-Einstein metrics \cite{HultgrenWittNystrom}. These are $k$-tuples of K\"ahler metrics $(\omega_1,\ldots,\omega_k)$ on a compact K\"ahler manifold $X$ satisfying
\begin{equation}
    \Ric \omega_1 = \ldots = \Ric \omega_k = \pm\sum_i \omega_i.
    \label{eq:cKE} 
\end{equation}
These generalizes K\"ahler-Einstein metrics in the sense that that for $k=1$ this equation reduces to the classical equation
$$ \Ric \omega_1 = \pm \omega_1 $$
defining K\"ahler-Einstein metrics. Moreover, \eqref{eq:cKE} implies a cohmological condition on $\omega_1,\ldots,\omega_k$, namely
\begin{equation}
    \label{eq:CohomologicalCondition}
    \sum_i [\omega_i] = \pm c_1(X).
\end{equation}
We see that, similarly as for K\"ahler-Einstein metrics, the theory splits into two cases: $c_1(X)<0$ and $c_1(X)>0$. Now, as in \cite{HultgrenWittNystrom} we will say that a $k$-tuple of K\"ahler classes $(\alpha_1,\ldots,\alpha_k)$ such that $\sum_i \alpha_i=\pm c_1(X)$ is a \emph{decomposition of $\pm c_1(X)$} and given a decomposition of $c_1(X)$ we will say that it admits a coupled K\"ahler-Einstein metric if there is a coupled K\"ahler-Einstein metric $(\omega_1,\ldots,\omega_k)$ such that $[\omega_i]=\alpha_i$ for all $i$. In \cite{HultgrenWittNystrom} it was shown that fixing a decomposition of $c_1(X)$ imposes the right boundary conditions on \eqref{eq:cKE} in the sense that:
\begin{itemize}
    \item If $c_1(X)<0$, then any decomposition of $-c_1(X)$ admits a unique coupled K\"ahler-Einstein metric. 
    \item If $c_1(X)>0$, then any coupled K\"ahler-Einstein metric admitted by a given decomposition of $c_1(X)$ is unique up to the flow of holomorphic vector fields.
\end{itemize}
Moreover, it was shown that if $c_1(X)>0$ and $(\omega_1,\ldots,\omega_k)$ is a coupled K\"ahler-Einstein metric, then the associated $k$-tuple of K\"ahler classes $([\omega_1],\ldots,[\omega_k])$ satisfies a certain algebraic stability condition which, by analogy, was called K-polystability. It was also conjectured that the converse of this holds, providing a ''coupled'' Yau-Tian-Donaldson conjecture:
\begin{conjecture}\cite{HultgrenWittNystrom}
Assume $c_1(X)>0$. Then a decomposition of $c_1(X)$ admits a coupled K\"ahler-Einstein metric if and only if it is K-polystable. 
\end{conjecture}

Our main theorem confirms this conjecture in the toric case and provides a simple condition for K-polystability in terms of the barycenters of a collection of polytopes associated to $(\alpha_1,\ldots,\alpha_k)$. More precisely, consider the anti-canonical line bundle $-K_X$ over a toric Fano manifold $X$. Fixing the action of $(\C^*)^n$ on $X$, this defines a polytope $P_{-K_X}$ in the vector space $M\otimes \R$ where $M$ is the character lattice of $(\C^*)^n$. For a general K\"ahler class that arise as the curvature of a toric line bundle, this correspondence is well defined up to translation of the polytope (or equivalently, up to choice of action on the toric line bundle). Moreover, the correspondence trivially extends to all K\"ahler classes that can be written as linear combinations with positive real coefficients of K\"ahler classes of this type. By general facts (see Lemma~\ref{lemma:RClasses} and the discussion following it) this holds for any K\"ahler class on a toric Fano manifold. 
%
This means that a decomposition of $c_1(X)$ determines (up to translations) a set of polytopes $P_1,\ldots,P_k$ in $\R^n$. Moreover, the condition $\sum_i \alpha_i = c_1(X)$ means the polytopes can be chosen so that the Minkowski sum
\begin{equation}
    \sum_i P_i = P_{-K_X}.
    \label{eq:PolytopeNormalization}
\end{equation}
Enforcing this, we note that the polytopes associated to a decomposition of $c_1(X)$ are well defined up to translations 
$$ (P_1,\ldots,P_k) \mapsto (P_1+c_1,\ldots,P_k+c_k)$$ 
where $c_1,\ldots,c_k\in \R^n$ satisfies $\sum_i c_i=0$.

Now, given a polytope $P$ in $\R^n$ we will let $b(P)$ be the (normalized) barycenter of $P$
$$ b(P) = \frac{1}{\Vol(P)}\int_P p dp $$
where $dp$ is the uniform measure on $P$ and $\Vol(P)=\int_Pdp$. Note that $b(P+c) =b(P)+c$, hence, assuming \eqref{eq:PolytopeNormalization}, the quantity $\sum_i b(P_i)$ is independent of the choices of translation of $P_1,\ldots,P_k$.  
Our main theorem is:
\begin{theorem}
    \label{thm:cYTD}
    Let $X$ be a toric Fano manifold. Assume  $(\alpha_i)$ is a decomposition of $c_1(X)$ and $P_1,\ldots,P_k$ are the associated polytopes. Then the following is equivalent:
    \begin{itemize}
    \item $(\alpha_i)$ admits a coupled K\"ahler-Einstein tuple
    \item $(\alpha_i)$ is K-polystable in the sense of \cite{HultgrenWittNystrom}
    \item $ \sum_i b(P_i) = 0 $
    \end{itemize}
\end{theorem}
\begin{remark}
One important point is $\sum_i b(P_i)$ is not in general equal to
$$ b\left(\sum_i P_i\right) = b(P_{-K_X}), $$
hence the condition on $P_1,\ldots,P_k$ in Theorem~\ref{thm:cYTD} is not (a priori) equivalent to existence of a classical K\"ahler-Einstein metric. In fact, non of these conditions imply the other. By Corollary~\ref{cor:example} below, there is an example of a manifold that don't admit Kähler-Einstein metrics but do admit coupled Kähler-Einstein metrics. Moreover, by Remark~\ref{rem:NonStableDecomposition} there is an example of a Kähler-Einstein manifold with decompositions of $c_1(X)$ that don't admit coupled Kähler-Einstein metrics.
\end{remark}
\begin{corollary}
\label{cor:example}
Let $E$ be the rank 2 vector bundle
$$ E = \mathcal O_{\mathbb P^2}(-1) \oplus \mathcal O_{\mathbb P^1}(-1) $$
over $\mathbb P^2 \times \mathbb P^1$ and consider the toric four-manifold $X=\mathbb P(E)$. Then $X$ does not admit a Kähler-Einstein metric. 
On the other hand, let $\pi:X\rightarrow \mathbb P^1$ be the natural projection onto $\mathbb P^1$ and $\beta_1, \beta_2 \in H^{(1,1)}(X)$ be the classes corresponding to the divisors given by $\pi^{-1}(0)$ and $\pi^{-1}(\infty)$, respectively. 
Then
\begin{equation}
    \label{eq:ExampleDecomposition}
    \alpha_1 = \frac{c_1(X)}{2} - \frac{\sqrt{\frac{5}{7}}(\beta_1+\beta_2)}{4}, \, \alpha_2 =  \frac{c_1(X)}{2} + \frac{\sqrt{\frac{5}{7}}(\beta_1+\beta_2)}{4}
\end{equation}
are K\"ahler and the decomposition of $c_1(X)$ given by $(\alpha_1,\alpha_2)$
admits a coupled K\"ahler-Einstein metric.
\end{corollary}
\begin{remark}
It would be interesting to see if there are simpler examples than the one given in Corollary~\ref{cor:example} of manifolds which admit coupled Kähler-Einstein metrics but no Kähler-Einstein metrics. However, when attending this it is important to note that, by Corollary~1.6 in \cite{HultgrenWittNystrom}, the automorphism group of any manifold that admits a coupled K\"ahler-Einstein metric is reductive. Among other things, this rules out $\mathbb P^2$ blown up in one or two points. 
\end{remark}
\begin{remark}
\label{rem:NonStableDecomposition}
The following is an example of a decomposition of $c_1(X)$ on an Einstein manifold that does not admit a coupled K\"ahler-Einstein metric. Let $X$ be the toric Fano manifold acquired by blowing up $\mathbb P^2$ in three points and $D$ be the $(S^1)^n$-invariant divisor in $X$ that corresponds to the ray generated by $(1,1)$ in the fan of $X$. Let $D_t = -K_X/2 + tD$. We have $D_t+D_{-t}=-K_X$. Computer calculations shows that 
$$ b\left(P_{D_t}\right)+b\left(P_{D_{-t}}\right) \not= 0 $$
for small $t$, in other words the decomposition of $c_1(X)$ given by $(c_1(D_t),c_1(D_{-t}))$ does not admit a coupled K\"ahler-Einstein metric for small $t$.
\end{remark}
\begin{remark}
As discussed in \cite{HultgrenWittNystrom}, fixing a K\"ahler class $\alpha$ on $X$ we get a family of decompositions of $c_1(X)$
$$ \{(t\alpha,c_1(X)-t\alpha): t\in (0,t_\alpha) \}, $$
where $t_\alpha = \sup\{t:c_1(X)-t\alpha>0\}$. Assuming they admit coupled K\"ahler-Einstein metrics $(\eta_1^t,\eta_2^t)$ we get a canonical family of metrics $\{\omega_t := \eta_1^t/t\}$ in $\alpha$. Now, let $X$ be a toric Fano surface. By Theorem~\ref{thm:cYTD}, $(t\alpha,\alpha-c_1(X))$ admits a coupled K\"ahler-Einstein metric if and only if
\begin{equation}
    \label{eq:CanonicalFamilyCondition}
    tb(P_{L_\alpha}) + b(P_{-K_X-tL_\alpha}) = 0 
\end{equation}
where $L_\alpha$ is a toric ($\R$-)line bundle such that $c_1(L_\alpha)=\alpha$. On the other hand, it was proven in \cite{Donaldson09} that $\alpha$ admits a constant scalar curvature metric if and only if 
\begin{equation}
    \label{eq:cscKCondition}
    b(P_{L_\alpha}) - \frac{1}{\int_{\partial P_{L_\alpha}}d\sigma}\int_{\partial P_{L_\alpha}} p d\sigma = 0 
\end{equation}
where $d\sigma$ is the measure on $\partial P_{L_\alpha}$ defined by the identity
$$ \left.\frac{d}{dt} \left(\int_{P_{L_\alpha}+tP_{-K_X}} fdp\right)\right|_{t=0} = \int_{\partial P_{L_\alpha}} f d\sigma $$
for all functions $f$ continuous in a neighbourhood of $P$. It would be interesting to understand the relationship between the conditions \eqref{eq:CanonicalFamilyCondition} and \eqref{eq:cscKCondition}. 
\end{remark}

Our second result considers a more general (soliton type) version of \eqref{eq:cKE}, namely, given holomorphic vector fields $V_1,\ldots,V_k$ 
\begin{equation}
    \Ric \omega_1 -L_{V_1}(\omega_1) = \ldots = \Ric \omega_k - L_{V_k}(\omega_k) = \sum_i \omega_i. 
    \label{eq:cKESoliton}
\end{equation}
We will say that a $k$-tuple of K\"ahler metrics satsifying \eqref{eq:cKE} is a \emph{coupled K\"ahler-Ricci soliton}. When $k=1$, \eqref{eq:cKESoliton} reduces to \eqref{eq:ClassicalSoliton} and defines classical K\"ahler-Ricci solitons. As mentioned above these appear as natural solutions to the K\"ahler-Ricci flow. In fact, a similar interpretation in terms of natural solutions to a geometric flow can be given for \eqref{eq:cKESoliton}. Given $k$ K\"ahler metrics $\omega_1^0,\ldots,\omega_k^0$ we may consider the flow defined by 
\begin{eqnarray} 
    \frac{d}{dt} \omega_1^t & = & \Ric\omega_1^t - \sum_i \omega_i^t \nonumber \\
    & \vdots & \nonumber \\
    \frac{d}{dt} \omega_k^t & = & \Ric\omega_k^t - \sum_i \omega_i^t, \label{eq:RicciFlow}
\end{eqnarray}
for $t\in [0,\infty)$. Stationary solutions to \eqref{eq:RicciFlow} are given by coupled K\"ahler-Einstein metrics, i.e. solutions to \eqref{eq:cKE}. On the other hand, putting $V_1=\ldots=V_k=V$ and letting $(\omega_i^t)$ be the flow along $V$ of a $k$-tuple $(\omega_i^0)$ satisfying \eqref{eq:cKESoliton} means $(\omega_i^t)$ will satisfy \eqref{eq:cKESoliton} for each $t$.  Plugging this into the right hand side of \eqref{eq:RicciFlow} gives 
$$ \Ric\omega_j^t - \sum_i \omega_i^t = L_{V} \left(\omega_j^t\right) $$
for all $j$. By definition 
$ \frac{d}{dt} \omega_j^t = L_{V} (\omega_j^t) $ for all $j$, hence $ (\omega_i^t) $  satisfies \eqref{eq:RicciFlow}.

To state our second result we need some terminology. Note that a point in the vector space that is dual to $M\otimes \R$, namely $N\otimes \R$ where $N$ is the lattice consisting of one parameter subgroups in $(\C^*)^n$, determines a holomorphic vector field on $X$. We will call any holomorphic vector field on $X$ that arise in this manner a \emph{toric vector field}. These can be given a concrete description in the following way: By definition, the action of $(\C^*)^n$ on $X$ admits a open, dense and free orbit. Identifying $(\C^*)^n$ with this orbit and letting $\sigma_1,\ldots,\sigma_n$ be the standard logarithmic coordinates on $(\C^*)^n$ the toric vector fields are simply the vector fields that arise as linear combinations of the coordinate vector fields $\frac{\partial}{\partial \sigma_1},\ldots,\frac{\partial}{\partial \sigma_k}$. We will often identify a toric vector field with its associated point in $N\otimes \R$. 

In this context there is a natural vector valued invariant $\mathcal A_V(P)$ determined by a polytope $P$ in $\R^n=M\otimes \R$ and a point $V$ in the dual vector space $N\otimes \R$. To define it we first introduce the $V$-weighted volume of $P$
$$ \Vol_V(P) = \int_P e^{\langle V,p \rangle} dp. $$
Then $\mathcal A_V(P)$ is given by
\begin{equation}
    \mathcal A_{P}(V) = \frac{1}{\Vol_V(P)}\int_{P} p e^{\langle V,p\rangle}dp.
    \label{eq:DefA}
\end{equation}
With respect to this we have:
\begin{theorem}
    \label{thm:MainTheorem}
    Let $V_1,\ldots,V_k$ be toric vector fields on a toric Fano manifold $X$. Assume  $(\alpha_1,\ldots,\alpha_k)$ is a decomposition of $c_1(X)$ and $P_1,\ldots,P_k$ are the associated polytopes. Then there is a $(S^1)^n$-invariant solution $(\omega_1,\ldots,\omega_k)$ to \eqref{eq:cKESoliton} such that $\omega_i\in \alpha_i$ for each $i$ if and only if 
    \begin{equation}
        \sum_i \mathcal{A}_{P_i}(V_i) = 0. 
        \label{eq:SolitonCondition}
    \end{equation}
\end{theorem}
\begin{remark}
Similarly as in Theorem~\ref{thm:cYTD}, the polytopes $P_1,\ldots,P_k$ associated to $(\alpha_1,\ldots,\alpha_k)$ are only well defined up to translations $P_i\rightarrow P_i+c_i$ for $c_i\in \R^n$ satisfying $\sum_i c_i = 0$. On the other hand, similarly as the barycenter, $\mathcal A_V(P)$ satisfies
$$ \mathcal A_{P+c}(V) = \mathcal A_{P}(V)+c, $$
hence the left hand side of \eqref{eq:SolitonCondition} is invariant under such translations.  
\end{remark}
\begin{remark}
Theorem~\ref{thm:MainTheorem} is a generalization of Wang and Zhu's theorem on existence of K\"ahler-Ricci solitons on toric manifolds \cite{WangZhu}. See also \cite{BermanBerndtsson} and \cite{Delcroix} for generalizations in other directions. 
\end{remark}

A straight forward corollary of Theorem~\ref{thm:MainTheorem}, using that \eqref{eq:DefA} is the gradient of a strictly convex and proper function on $\R^n$, is:
\begin{corollary}
    \label{cor:SpecialSoliton}
    Let $(\alpha_i)$ be a decomposition of $c_1(X)$ on a toric Fano manifold. Then there is a unique toric vector field $V$ such that $(\alpha_i)$ admits a $(S^1)^n$-invariant coupled K\"ahler-Ricci soliton where $V_1=\ldots=V_k=V$.
\end{corollary}
\begin{remark}
Naturally, we expect solutions of the flow \eqref{eq:RicciFlow} to converge to the K\"ahler-Ricci solitons in Corollary~\ref{cor:SpecialSoliton}. This parallels the theory in the case $k=1$ (see \cite{TianZhu07}). On the other hand, it is interesting to note that by Theorem~\ref{thm:MainTheorem} there exist a large class of solitons that does not appear as natural solutions to \eqref{eq:RicciFlow} in the sense discussed above (this happens whenever $V_i\not=V_j$ for some $i$ and $j$). This suggests that there is a more general flow, which includes \eqref{eq:RicciFlow} as a special case, and where the solitons of Theorem~\ref{thm:MainTheorem} appear as natural solutions. 
\end{remark}


A second corollary of Theorem~\ref{thm:MainTheorem} is related to the corresponding real Monge-Amp\`ere equation. Let $f_1,\ldots,f_k$ be twice differentiable convex functions on $\R^n$. Let $\nabla f_i$ denote the gradient of $f_i$. Then, given a decomposition $(\alpha_1,\ldots,\alpha_k)$ and associated polytopes $P_1,\ldots,P_k$, existence of coupled K\"ahler-Ricci solitons is equivalent to the solvability of the equation
\begin{equation} 
    \label{eq:RealMA}
    \frac{e^{\langle V_1,\nabla f_1\rangle}}{\Vol_{V_1}(P_1)} \det \left( \frac{d^2f_1}{dx_ldx_m} \right) = \ldots = \frac{e^{\langle V_k,\nabla f_k\rangle}}{\Vol_{V_k}(P_k)} \det \left( \frac{d^2f_k}{dx_ldx_m}\right) = e^{-\sum_i f_i}
\end{equation}
under the boundary conditions
\begin{equation}
    \overline{ \nabla f_i(\R^n) } = P_i
    \label{eq:BoundaryCondition}
\end{equation}
where the left hand side of \eqref{eq:BoundaryCondition} denotes the closure of the image of $\nabla f_i$ in $\R^n$. We will say that a $k$-tuple of polytopes in $\R^n$ is \emph{toric Fano} if it is defined by a decomposition of $c_1(X)$ on a toric Fano manifold. 
\begin{corollary}
    \label{cor:RealEquation}
    Assume $P_1,\ldots,P_k$ is a toric Fano $k$-tuple of polytopes and $V_1, \ldots, V_k \in \R^n$. Then \eqref{eq:RealMA} admits a solution satisfying \eqref{eq:BoundaryCondition} if and only if 
    $$ \sum_i \mathcal{A}_{P_i}(V_i) = 0. $$
    In particular, if $V_1=\ldots=V_k=0$ then \eqref{eq:RealMA} admits a solution satisfying \eqref{eq:BoundaryCondition} if and only if 
    $$ \sum_i b(P_i) = 0. $$
\end{corollary}

Theorem~\ref{thm:cYTD} essentially follows from considering the case $V_1=\ldots=V_k=0$ in Theorem~\ref{thm:MainTheorem}. Doing this gives that the third point in Theorem~\ref{thm:cYTD} implies the first point. As mentioned above, by a previous result (Theorem~1.15 in \cite{HultgrenWittNystrom}) the first point implies the second point. Finally, an explicit formula for the (coupled) Donaldson-Futaki invariant of test configurations induced by toric vector fields shows that the second point implies the third point. To be more precise, if $V$ is a toric vector field and $(\alpha_i)$ is a decomposition of $c_1(X)$ with associated polytopes $P_1,\ldots,P_k$, then the test configuration for $(\alpha_i)$ induced by $V$ has Donaldson-Futaki invariant 
$$ \left\langle V,\sum_i b(P_i)\right\rangle. $$
It follows that if $\sum_i b(P_i)\not=0$, then there is a test configuration for $(\alpha_i)$ with negative Donaldson-Futaki invariant. By definition, this means $(\alpha_i)$ is not K-polystable (see Section~\ref{sec:TestConfigurations} for a detailed argument).

The main point in the proof of Theorem~\ref{thm:MainTheorem} is to establish a priori $C^0$-estimates along an associated continuity path. More precisely, let $\theta_1,\ldots,\theta_k$ be K\"ahler metrics such that $[\theta_i]=\alpha_i$. Assume, using the Calabi-Yau theorem, that $\omega_0$ is a K\"ahler form such that $\Ric \omega_0 = \sum_i\theta_i$ and $\int_X\omega_0^n = 1$. For each $i$, let $g_i=g_{\theta_i,V_i}$ be a $\theta_i$-plurisubharmonic function on $X$ such that 
$$ dd^c g_i = L_{V_i}(\theta_i) $$
and $\int_Xe^{g_i}\theta_i^n=1$ (see Lemma~\ref{lemma:SolitonPotential}). For $t\in [0,1]$ we will consider the equation 
\begin{equation}
e^{g_1+{V_1}(\phi_1)}\left(\theta_1+dd^c\phi_1\right)^n  = \ldots = e^{g_k+{V_k}(\phi_k)} \left(\theta_k+dd^c\phi_k\right)^n = e^{-t\sum_i \phi_i}\omega_0^n.
\label{eq:ContinuitySetup}
\end{equation}
Moreover, fixing a point $x_0\in X$ we will assume solutions to \eqref{eq:ContinuitySetup} are normalized according to
\begin{equation} 
    \phi_1(x_0) = \ldots = \phi_k(x_0).
    \label{eq:Normalization}
\end{equation}
The significance of this equations is that for $t=1$, a $k$-tuple of functions $\phi_1,\ldots,\phi_k$ such that each $\phi_i$ is $\theta_i$-plurisubharmonic solves \eqref{eq:ContinuitySetup} if and only if the $k$-tuple of K\"ahler metrics $(\theta_i+dd^c\phi_i)$ is a coupled K\"ahler Ricci soliton. 
We prove:
\begin{theorem}
    \label{thm:C0Estimates}
    Let $V_i$, $\alpha_i$ and $P_i$ be as in Theorem~\ref{thm:MainTheorem} and assume \eqref{eq:SolitonCondition} holds. Let $x_0$ be the point in $X$ that, under the identification of $(\C^*)^n$ with its open, dense and free orbit, corresponds to the identity element in $(\C^*)^n$. Then, for any $t_0>0$ there is a constant $C$ such that any solution $(\phi_1,\ldots,\phi_k)$ of \eqref{eq:ContinuitySetup} for $t\geq t_0$, normalized according to \eqref{eq:Normalization}, satisfies 
    $$ \sup_X |\phi_i| < C $$
    for all $i$.
\end{theorem}

In \cite{Pingali} Pingali reduces existence of coupled K\"ahler-Einstein metrics to a priori $C^0$-estimates. This means that Theorem~\ref{thm:MainTheorem} in the special case when $V_1=\ldots=V_k=0$, and thus Theorem~\ref{thm:cYTD}, follows from Theorem~\ref{thm:C0Estimates} above and Pingali's work. For the general case we adapt the argument of Pingali to the soliton setting, essentially following the computations by Tian and Zhu in \cite{TianZhu}. 
Letting $\Aut(X)$ be the automorphism group of $X$ we prove:
\begin{theorem}
    \label{thm:OpennessAndHigherOrderEstimates}
    Let $X$ be a Fano manifold and $V_1,\ldots,V_k$ be holomorphic vector fields in the reductive part of the Lie algebra of $\Aut(X)$ such that $\Im V_i$ generate a compact one parameter subgroup in $\Aut(X)$ for each $i$. Let $(\alpha_i)$ be a decomposition of $c_1(X)$ with representatives $\theta_1,\ldots,\theta_k$ such that $\Im L_{V_i}\theta_i = 0$ for all $i$. Assume also $C^0$-estimates hold for \eqref{eq:ContinuitySetup}, in other words, for each $t_0>0$, there is a constant $C$ such that any solution $(\phi_i)$ to \eqref{eq:ContinuitySetup} at $t>t_0$ satisfies
    $$ \sup_X |\phi_i| < C $$
    for all $i$. Then $(\alpha_i)$ admits a solution to \eqref{eq:cKESoliton}.
\end{theorem}

We get that the positive part of Theorem~\ref{thm:MainTheorem} follows directly from Theorem~\ref{thm:C0Estimates} and Theorem~\ref{thm:OpennessAndHigherOrderEstimates}. The negative part of Theorem~\ref{thm:MainTheorem} follows directly from a change of variables in \eqref{eq:RealMA} (see Lemma~\ref{lemma:Obstruction}).

\begin{remark}
In \cite{BermanBerndtsson} Berndtsson and Berman uses a variational approach to prove existence of K\"ahler-Ricci solitons on toric log Fano varieties. They give a direct argument for coercivity of the associated Ding functional on $(S^1)^n$-invariant metrics. It would be interesting if this coercivity estimate could be extended to the coupled setting. This would provide a stronger result than this paper in two respects: First of all, it would cover the singular setting of log Fano varieties. Secondly, since this bypasses the higher order a priory estimates from complex geometry it would provide a version of Corollary~\ref{cor:RealEquation} that is valid for all $k$-tuples of polytopes, not only the ones that are defined by decompositions of $c_1(X)$ on toric Fano manifolds. 
\end{remark}

This paper is organized in the following way: Section~\ref{sec:Openness} and Section~\ref{sec:HigherOrderEstimates} are devoted to the proof of Theorem~\ref{thm:OpennessAndHigherOrderEstimates}. In Section~\ref{sec:Openness} we prove openness along the continuity path and solvability at $t=0$. In Section~\ref{sec:HigherOrderEstimates} we prove $C^{2,\alpha}$-estimates assuming $C^0$-estimates, thus finishing the proof of Theorem~\ref{thm:OpennessAndHigherOrderEstimates}. In Section~\ref{sec:C0Estimates} we set up the real convex geometric framework and in Section~\ref{sec:Estimates} we use this to prove the $C^0$-estimate of Theorem~\ref{thm:C0Estimates}. Finally, at the end of Section~\ref{sec:Estimates} we prove Theorem~\ref{thm:MainTheorem}, Corollary~\ref{cor:SpecialSoliton} and Corollary~\ref{cor:RealEquation} and in Section~\ref{sec:TestConfigurations} we prove Theorem~\ref{thm:cYTD}.

\paragraph{Acknowledgements}
The author would like to thank David Witt Nystr\"om for many fruitful discussions relating to this paper and Thibaut Delcroix for his suggestion to allow the $k$ vector fields in equation \eqref{eq:cKESoliton} to be different from one another. Moreover, the author would like to thank Yanir Rubinstein for directing him to \cite{Futaki} when looking for an example of a manifold with nonzero Futaki invariant but reductive automorphism group. The latter suggestion led to Corollary~\ref{cor:example}. 

\section{Openness and higher order estimates}
\label{sec:OpennessAndHigherOrderEstimates}
This Section is devoted to proving Theorem~\ref{thm:OpennessAndHigherOrderEstimates}. 

The following lemma is well known. However, as a courtesy to the reader we include a proof of it. 
\begin{lemma}
    \label{lemma:SolitonPotential}
    Assume X is a Fano manifold, $V$ a holomorphic vector field on $X$ and $\theta$ a K\"ahler form on $X$ such that the imaginary part of $L_V(\theta)$ vanishes. Then there is a smooth real valued function $g$ on $X$ such that 
    $$ dd^c g = L_V(\theta). $$
\end{lemma}
\begin{proof}
Since $V$ is a holomorphic vector field, the contraction operator $i_V$ anti-commutes with $\bar\partial$, hence $i_V\theta$ is a $\bar\partial$-closed $(0,1)$-form. By the Kodaira Vanishing Theorem, since $X$ is Fano, the sheaf cohomology group 
$$H^1(X,\mathcal O)=H^1(X,-K_X+K_X) = 0.$$ 
This means the Dolbeault cohomology group 
$$H^{(0,1)}(X)\cong H^1(X,\mathcal O) = 0,$$
hence $i_V \theta_i$ is also $\bar \partial$-exact. Let $g$ be a smooth function such that $\sqrt{-1}\bar\partial g = i_V\theta$. As $L_V(\theta)$ is real, so is $g$. Moreover, $$ dd^cg = i\partial\bar\partial g = \partial i_V \theta = L_V(\theta). $$
This proves the lemma. 
\end{proof}

Now, let $\phi_i$ be a smooth function in $\Psh(X,\theta_i)$. Noting that 
$$ L_{V_i}(dd^c\phi) = \partial i_{V_i} \sqrt{-1}\partial \bar\partial \phi_i = \sqrt{-1}\partial\bar\partial i_{V_i}\partial  \phi_i = dd^c V_i(\phi_i) $$
we see that $dd^c(g_i+{V_i}(\phi_i)) = L_V(\theta_i+dd^c\phi_i)$. This means that, similarly as in \cite{HultgrenWittNystrom}, we get:
\begin{lemma}
\label{lemma:ComplexEquation}
Let $X$ be a Fano manifold, $V_1,\ldots,V_k$ holomorphic vector fields on $X$ and $(\alpha_i)$ a $k$-tuple of K\"ahler classes on $X$ such that $\sum \alpha_i = c_1(X)$. Assume each class $\alpha_i$ has a representative $\theta_i$ such that $\Im L_V (\theta_i)=0$. Then $(\phi_1,\ldots,\phi_k)$ is a solution to \eqref{eq:ContinuitySetup} at $t=1$ if and only if the $k$-tuple of K\"ahler metrics $(\theta_i+dd^c\phi_i)$ is a coupled K\"ahler-Ricci soliton. 
\end{lemma}

\subsection{Openness}
\label{sec:Openness}
Here we will prove the first part of Theorem~\ref{thm:OpennessAndHigherOrderEstimates}, namely that the set of $t$ such that \eqref{eq:ContinuitySetup} is solvable is open. 

We define the following Banach spaces 
$$ A= \left\{(\phi_i)\in \left(C^{4,\alpha}(X)\right)^k: \phi_1(x_0)=\ldots=\phi_k(x_0)\right\} $$
and 
$$ B = \left\{(v_i)\in \left(C^{2,\alpha}(X)\right)^k: \int_X v_1 \mu = \ldots =\int_X v_k \mu\right\}. $$ 

Let $F:\R\times A\rightarrow B\times \R^{k-1}$ be defined by
$$
F(t,(\phi_i)) = 
\begin{pmatrix} 
\log\frac{(\theta_1+dd^c\phi_1)^n}{\theta_1^n} + g_1+{V_1}(\phi_1) + t\sum \phi_i \\
\vdots \\
\log \frac{(\theta_k+dd^c\phi_k)^n}{\theta_k^n} +g_k+{V_k}(\phi_k) + t\sum \phi_i \\
\phi_2(x_0)-\frac{1}{k}\sum_i \phi_i(x_0) \\
\vdots \\
\phi_1(x_0)-\frac{1}{k}\sum_i \phi_i(x_0)
\end{pmatrix}.
$$
Note that $F(t,(\phi_i))=0$ if and only if $(\phi_i)$ defines a normalized solution to \eqref{eq:ContinuitySetup} at $t$. Moreover, in this case the measure
$$ \mu := \left(\theta_i+dd^c\phi_i\right)^n e^{g_i+{V_i}(\phi_i)} $$ 
is independent of $i$.

\begin{lemma}
\label{lemma:Linearization}
The linearization of $F$ at $(t,\phi)$ with respect to the second argument is given by $H:A\rightarrow B\times \R^{k-1}$ defined by
\begin{equation}
\label{eq:Linearization}
H(v_1,\ldots,v_k) = \begin{pmatrix} 
-\Delta_{\omega_1} v_1 + {V_1}(v_1) + t\sum v_i \\
\vdots \\
-\Delta_{\omega_k} v_k + {V_k}(v_k) + t\sum v_i \\
v_2(x_0)-\frac{1}{k}\sum_i v_i(x_0) \\
\vdots \\
v_k(x_0)-\frac{1}{k}\sum_i v_i(x_0) 
\end{pmatrix}.
\end{equation}
where $\omega_i=\theta_i+dd^c\phi_i$ and $\Delta_{\omega_i}$ is the associated Laplace-Beltrami operator. Moreover, $H$ is elliptic. Finally, assume $F(t,\phi)=0$ and let $\langle\cdot,\cdot\rangle$ be the inner product on $(C^{2,\alpha}(X))^k$ given by
$$ \langle(u_i),(v_i)\rangle = \sum_i \int_X u_iv_i d\mu $$
Then $\langle H(u_1,\ldots,u_k) , (v_i)\rangle = \langle (u_i),H(v_1,\ldots,v_k) \rangle$ for any $(u_i),(v_i)\in (C^{2,\alpha})^k$.
\end{lemma}
\begin{proof}
Equation \eqref{eq:Linearization} follows from straight forward differentiation and the well known identity
$$ \left.\frac{d}{ds}\log\frac{(\theta_i+dd^c(\phi_i+sv_i))^n}{\theta_i^n}\right|_{s=0} = 
n\frac{dd^cv(\theta_i+dd^c\phi_i)^{n-1}}{(\theta_i+dd^c\phi_i)^n} = \Delta_{\omega_i}v_i. $$


Now, $H$ takes the following form in local coordinates:
$$ (u_i)\mapsto (v_j) = \left(\sum_{i,l,m} a^{lm}_{ij}(x)\frac{\partial^2u_i}{\partial x_l\partial x_m} + \textnormal{lower order terms}\right) $$
where $a_{ij}^{lm}=0$ if $i\not=j$ and $\{a_{ii}^{lm}\}_{l,m}$ are the coefficients for the Laplace operator $\Delta_{\omega_i}$. Recall that $H$ is elliptic if the matrix
\begin{equation} \left(\sum_{l,m}a^{lm}_{ij}(x)\xi_l\xi_m\right) \label{eq:EllipticityMatrix} \end{equation}
is invertible for all $p\in X$ and all non-zero $\xi=\sum \xi_l\frac{\partial}{\partial x_l}\in T_pX$. But this follows immediately. To see this note that 
$$\sum_{l,m}a^{lm}_{ij}(x)\xi_l\xi_m$$
is 0 if $i\not= j$ and, by ellipticity of $\Delta_{\omega_i}$, positive if $i=j$. This means \eqref{eq:EllipticityMatrix} is a diagonal matrix with positive entries on the diagonal, hence it is invertible.  

We will now prove the last statement in the lemma. It is a consequence of the following identity for functions $u,v\in C^{2,\alpha}(X)$ (see Lemma~2.2 in \cite{TianZhu}):
\begin{equation} \int_X (\Delta_{\omega_i} v + {V_i}(v))u d\mu = - \int_X \langle dv,du\rangle_{\omega_i} d\mu. \label{eq:PartialIntegration} \end{equation}
We get
\begin{eqnarray}
& & \sum_i \int_X \left(\Delta_{\omega_i}v_i + {V_i}(v_i) + \sum_j v_j\right)u_i d\mu \nonumber \\
& = & -\sum_i \int_X \langle dv_i,du_i \rangle d\mu + \sum_{i,j} \int_X v_ju_id\mu \nonumber \\
& = & \sum_i \int_X v_i\left(\Delta_{\omega_i}u_i + {V_i}(u_i) + \sum_j u_j\right) d\mu, \nonumber
\end{eqnarray}
and the last statement in the lemma follows. 
\end{proof}

\begin{lemma}
\label{lemma:HInjective}
Assume $t\in [0,1)$ and $(v_i)\in (C^{4,\alpha}(X))^k$ are not all constant and satisfies
\begin{equation}
    \Delta_{\omega_1} v_1 + V_i(v_1) = \ldots = \Delta_{\omega_k} v_k + V_i(v_k) = \lambda\sum_i v_i
    \label{eq:LaplaceEigenfunction}
\end{equation}
for a $k$-tuple $\omega_1,\ldots,\omega_k$ satisfying
\begin{equation} 
\Ric \omega_1 - L_V(\omega_1) = \ldots = \Ric \omega_k - L_V(\omega_k) = t\sum_i \omega_i + (1-t)\sum_i \theta_i.
\label{eq:RicContinuityMethod}
\end{equation}
Then $\lambda > t$.
\end{lemma}
\begin{proof}
Let $\partial_{\omega_i}v$ denote the gradient of $v$ with respect to the metric $\omega_i$. Moreover, will use the notation $\Ric_{\omega_i} = \Ric(\omega_i)$. The proof is based on the following Weitzenböck identity (see \cite{TianZhu}, equation 2.7, page 277):
$$ -\int_X \langle d(\Delta_{\omega_i} v+ V_i(v)),dv\rangle_{\omega_i} \mu \geq \int_X (\Ric_{\omega_i}-L_V(\omega_i))\left(\partial_{\omega_i} v,\overline{\partial_{\omega_i} v}\right)d\mu. $$
Combining this with \eqref{eq:RicContinuityMethod} and \eqref{eq:LaplaceEigenfunction} gives
\begin{eqnarray}
\lambda^2\int_X \left(\sum_j v_j\right)^2 d\mu & = & \int_X (\Delta_{\omega_i} v_i+V_i(v_i))^2  d\mu \nonumber \\
& = & -\int_X \langle d(\Delta_{\omega_i} v_i+V_i(v_i)),dv\rangle_{\omega_i} d\mu \nonumber \\
& \geq & \int_X (\Ric_{\omega_i}+L_V(\omega_i))\left(\partial_{\omega_i} v_i,\overline{\partial_{\omega_i} v_i}\right) d\mu \nonumber \\
& \geq & t\int_X \sum_j |\partial_{\omega_i} v_i|_{\omega_j}^2 d\mu.
\label{eq:EigFirstPart}
\end{eqnarray}
Moreover, we claim that \eqref{eq:LaplaceEigenfunction} implies
\begin{equation}
\label{eq:ClaimNorms} 
\int_X |\partial_{\omega_i} v_i|_{\omega_j}^2 d\mu \geq \int_X |d v_j|_{\omega_j}^2 d\mu
\end{equation}
for any $i$ and $j$. Assuming that this is true we see that \eqref{eq:EigFirstPart} implies 
\begin{eqnarray}
\lambda^2\int_X \left(\sum_j v_j\right)^2 d\mu & \geq & t\int_X \sum_j |\partial_{\omega_j} v_j|_{\omega_j}^2 d\mu \nonumber \\
& = & t\int_X \sum_j |d v_j|_{\omega_j}^2 d\mu \nonumber \\
& = & t\int_X \sum_j\left(\Delta_{\omega_j} v_j + V_i(v_j)\right)v_j d\mu \nonumber \\
& = & t\lambda\int_X \sum_j\left(\sum_i v_i\right)v_j d\mu \nonumber \\
& = & t\lambda\int_X \left(\sum_j v_j\right)^2d\mu. \nonumber
\end{eqnarray}
We conclude that $\lambda\geq t$. Moreover, if $\lambda=t$ then equality holds in all inequalities above. In particular, equality holds in the last inequality of \eqref{eq:EigFirstPart}, hence, by \eqref{eq:RicContinuityMethod}, 
\begin{eqnarray}
0 & = & \int_X (\Ric_{\omega_i}-L_V(\omega_i))\left(\partial_{\omega_i} v_i,\overline{\partial_{\omega_i} v_i}\right) d\mu - t\int_X \sum_j |\partial_{\omega_i} v_i|_{\omega_j}^2 d\mu \nonumber \\
& = & (1-t)\int_X \sum_j|\partial_{\omega_i} v_i|_{\theta_j}^2 d\mu \nonumber
\end{eqnarray}
from which it follows that $v_i$ is constant for every $i$. This means that to finish the proof of the lemma it suffices to prove \eqref{eq:ClaimNorms}. To do this, note that for any $i$ and $j$, by \eqref{eq:LaplaceEigenfunction} 
\begin{eqnarray} \int_X |d v_j|_{\omega_j}^2 d\mu & = & \int_X (\Delta_{\omega_j} v_j+V_i(v_j))v_j d\mu \nonumber \\
& = & \int_X (\Delta_{\omega_i} v_i+V_i(v_i))v_j d\mu \nonumber \\
& = & \int_X \langle d v_i, d v_j \rangle_{\omega_i}.d\mu \nonumber
\end{eqnarray}
Moreover, choosing coordinates $(z_1,\ldots,z_n)$ that are normal with respect to $\omega_j$ and such that $\omega_i$ is diagonal with eigenvalues $\beta_1,\ldots,\beta_n$ at a point $p$ we get
\begin{eqnarray}
\left|\langle d v_i, d v_j \rangle_{\omega_i}\right| & = & \left| \sum_l \frac{1}{\beta_l} \frac{\partial v_i}{\partial z_l}\overline{\frac{\partial v_j}{\partial z_l}}\right| \nonumber \\
& \leq & \sqrt{\sum_l \left|\frac{1}{\beta_l}\frac{\partial v_i}{\partial z_l}\right|^2} \sqrt{\sum_l \left|\frac{\partial v_j}{\partial z_l}\right|^2} \nonumber \\
& = & |\partial_{\omega_i}v_i|_{\omega_j}|dv_j|_{\omega_j}. \nonumber 
\end{eqnarray}
Combining this with the Cauchy-Schwarz inequality we get
\begin{eqnarray}
\int_X |dv_j|_{\omega_j}^2 d\mu & = & \int_X \langle dv_i,dv_j\rangle_{\omega_i} d\mu \nonumber \\
& \leq & \int_X |\partial_{\omega_i}v_i|_{\omega_j}|dv_j|_{\omega_j} d\mu \nonumber \\
& \leq & \sqrt{\int_X |\partial_{\omega_i}v_i|_{\omega_j}^2d\mu} \sqrt{\int_X |dv_j|_{\omega_j}^2d\mu}, \nonumber
\end{eqnarray}
and \eqref{eq:ClaimNorms} follows.
\end{proof}
We can now prove the first part of Theorem~\ref{thm:OpennessAndHigherOrderEstimates}. 

\begin{proof}[Proof of Theorem~\ref{thm:OpennessAndHigherOrderEstimates}. First part: Openness and the case $t=0$]
The theorem is proved using the continuity method along the path defined by \eqref{eq:ContinuitySetup}. Here we will prove that the set of $t$ such that \eqref{eq:ContinuitySetup} is solvable is nonempty and open in $[0,1]$. At the end of Section~\ref{sec:HigherOrderEstimates} we will prove that it is also closed in $[0,1]$, hence that \eqref{eq:ContinuitySetup} is solvable for all $t\in [0,1]$.

First of all, to see that the set of $t$ such that \eqref{eq:ContinuitySetup} is solvable is nonempty, note that for $t=0$, \eqref{eq:ContinuitySetup} reduces to the collection of equations 
\begin{equation} 
\label{eq:Time0}
\left(\theta_j+dd^c\phi_j\right)^n e^{g_j+V_j(\phi_j)} = \omega_0^n. 
\end{equation}
This means that for each $i$  we can apply the Main Theorem in \cite{Zhu} to get $\phi_i$ such that 
\begin{equation}
    \label{eq:FromSolitanCY}
    \left(\theta_j+dd^c\phi_j\right)^n e^{g_j+V_j(\phi_j)+c_j} = \omega_0^n
\end{equation}
for some $c_j\in \R$. Integrating both sides of this and using the fact that 
\begin{equation}
    \label{eq:WeightedVolumeInvariant} 
    \int_X e^{g_i+V_i(\phi_i)}\left(\theta_i+dd^c\phi_i\right)^n = \int_X e^{g_i}\theta_i^n = 1 = \int_X \omega_0^n
\end{equation}
for all smooth $\phi_i\in \Psh(\theta_i)$ we see that $c_j=0$ for all $j$, in other words $(\phi_1,\ldots,\phi_k)$ provides a solution to \eqref{eq:ContinuitySetup} at $t=0$. 

Now, \eqref{eq:WeightedVolumeInvariant} is well known but for completeness we provide an argument for it here. Consider the variation of the left hand side of \eqref{eq:WeightedVolumeInvariant} with respect to $\phi_i$
\begin{equation}
    \label{eq:VariationOfWeightedVolume}
    \int_X \left(\Delta_{\omega_i} \dot\phi_i+V(\dot \phi_i)\right) \mu_i
\end{equation}
where we use the notation $\mu_i=e^{g_i+V_i(\phi_i)}\left(\theta_i+dd^c\phi_i\right)^n$. 
By \eqref{eq:PartialIntegration},
\begin{eqnarray}
\int_X \left(\Delta_{\omega_i} \dot\phi_i+V(\dot \phi_i)\right) (\dot\phi_i+1)\mu_i = \int_X|d\dot\phi|^2_{\omega_i}\mu_i = \int_X \left(\Delta_{\omega_i} \dot\phi_i+V(\dot \phi_i)\right) \dot\phi_i\mu_i,
\end{eqnarray}
hence \eqref{eq:VariationOfWeightedVolume} vanishes. This proves \eqref{eq:WeightedVolumeInvariant}. 

To prove that the set of $t$ such that \eqref{eq:ContinuitySetup} is solvable is open we will apply the Implicit Function Theorem. To do this we need to verify that the linearization $H$ of $F$ is invertible. By standard theory for elliptic partial differential equations this follows from Lemma~\ref{lemma:Linearization} and Lemma~\ref{lemma:HInjective}. More precisely, $H$ is elliptic by Lemma~\ref{lemma:Linearization}. This means the image of $H:W^{2,2}\rightarrow L^2$ is closed (see for example Theorem~10.4.7 in \cite{Nicolaescu}). Taking $(v_i)$ in the orthogonal complement of the image of $H$ gives 
$$\left\langle (v_i),H(u_i)\right\rangle = 0$$
for all $(u_i)\in (W^{2,2}(X))^k$. In particular, it holds for all $(u_i)\in (C^\infty(X))^k$. By the last point in Lemma~\ref{lemma:Linearization} this means $H(v_i)=0$ as a distribution. By elliptic regularity (see for example Corollary~10.3.10 in \cite{Nicolaescu}) $(v_i)\in (C^{\infty})^k$ and hence, by Lemma~\ref{lemma:HInjective}, $(v_i)=(C_i)$ such that $\sum C_i = 0$. Using elliptic regularity again (see for example Theorem~10.3.11b in \cite{Nicolaescu}) we see that $H$ is invertible from $C^{4,\alpha}(X)^k$ to $C^{2,\alpha}(X)^k$.
\end{proof}

\subsection{Higher order estimates}
\label{sec:HigherOrderEstimates}
We begin with 
\begin{lemma}
\label{lemma:LaplacianEstimate}
Assume $(\phi_i)$ satisfies $\eqref{eq:ContinuitySetup}$ for some $t\in [0,1]$. Then  
$$ \sup_X |\Delta_{\theta_j}\phi_j| \leq C $$
where $C$ depends only on $\sup_i ||\phi_i||_{C^0(X)}$.
\end{lemma}
We will use the following lemma from \cite{Zhu} (page 768, Corollary 5.3):
\begin{lemma}
\label{lemma:VFbound}
Let $X$ be a compact K\"ahler manifold, $\omega$ a K\"ahler form on $X$ and $V$ a holomorphic vector field on $X$. Assume $\phi\in \Psh(X,\omega)$ is smooth and $X(\phi)$ is a real-valued function. Then 
$$ \sup_X|V(\phi)|<C $$
for a constant $C$ that is independent of $\phi$. 
\end{lemma}
\begin{proof}[Proof of Lemma~\ref{lemma:LaplacianEstimate}]
We start with the following inequality originating in \cite{Yau} (see for example equation 2.3 on page 1587 in \cite{ChenHe}): Assume $\omega$ is a K\"ahler form and $v$ is a smooth function satisfying 
$$(\omega+dd^cv)^n = e^F\omega^n.$$ 
Then there are constants $C_1,C_2$ and $C_3$, independent of $v$, such that
\begin{equation}
\label{eq:YauEstimate}
\Delta_{\omega+dd^cv}\left(e^{-C_1v}(n+\Delta_\omega v)\right) \geq e^{-C_1v}\Delta_\omega F + C_2(n+\Delta_\omega v)^{\frac{n}{n-1}}-C_3. \end{equation}
For each $j$, we have that $\phi_j$ satisfies the equation
\begin{equation}
(\theta_j+dd^c\phi_j) = e^{-g_j-V_i(\phi_j)-t\sum_i \phi_i+\log(\omega_0^n/\theta_j^n)}\theta_j^n.
\label{eq:theta_iEquation}
\end{equation}
Applying \eqref{eq:YauEstimate} to this and letting 
$$ u_j = e^{-C_1\phi_j}(n+\Delta_{\theta_j}\phi_j), $$
for all $j$ we get
\begin{eqnarray} \Delta_{\omega_j} u_j & \geq & e^{-C_1\phi_j}\Delta_{\theta_j}\left( -g_j-V_j(\phi_j)-t\sum_i \phi_i+\log(\omega_0^n/\theta_j^n)\right) \nonumber \\
& & + C_2(n+\Delta_{\theta_j}\phi_j)^{\frac{n}{n-1}} - C_3. 
\label{eq:theta_iEquation2}
\end{eqnarray}
Note that $dd^c\phi_i>-\theta_i$, hence 
$$ \Delta_{\theta_j}\phi_j = n\frac{(dd^c\phi_j) \wedge \theta_j^{n-1}}{\theta_j^n} > -n. $$
This means $u_j > 0$ for all $j$. Moreover, $u_j - e^{-C_1\phi_j}\Delta_{\theta_j}\phi_j = ne^{-C_1\phi_j}$.
Hence, adjusting $C_2$ and $C_3$ in a way which only depends on $\sup_i||\phi_i||_{C^0(X)}$, we get 
\begin{equation} \Delta_{\omega_j} u_j \geq -e^{-C_1\phi_j}\Delta_{\theta_j}(g_j+V_j(\phi_j))-t\sum_i u_i + C_2u_j^{\frac{n}{n-1}} - C_3. 
\label{eq:theta_iEquation3}
\end{equation}

Now, let $V_j=\sum V_m^j \frac{\partial}{\partial z_m}$ and $\theta_j = \sum \theta^j_{m\bar l} dz_md\bar z_l$. As in \cite{TianZhu}, we compute
\begin{eqnarray} 
\Delta_{\theta_j}(g_j+V_j(\phi_j)) & = & \sum_{m,l}\frac{\partial}{\partial z_l}\left(V_m^j\left( \theta^j_{m\bar l}+\frac{\partial\phi_j}{\partial z_m \partial\bar z_l}\right)\right) \nonumber \\
& = & \sum_{m,l} \frac{\partial V_m^j}{\partial z_l} \left(\theta^j_{m\bar l}+\frac{\partial^2\phi_j}{\partial z_m \partial\bar z_l}\right) +V_m^j \left(\frac{\partial \theta^j_{m\bar l}}{\partial z_l} + \frac{\partial^3\phi_j}{\partial z_m\partial z_l \partial\bar z_l}\right). \nonumber \\
\label{eq:TianComp} 
\end{eqnarray}
We will be interested in this at a point, $p$, where $u_j$ attains its maximum. Choosing coordinates around $p$ that are normal with respect to $\theta_j$ and such that $\omega_j=\theta_j+dd^c\phi_j$ is diagonal, \eqref{eq:TianComp} reduces to
$$ \sum_{m} \frac{\partial V_m^j}{\partial z_m} \left(1+\frac{\partial^2\phi_j}{\partial z_m \partial\bar z_m}\right) +V_j(\Delta \phi_j). $$ 
The first term of this can be bounded by 
$$ \sup_m \left|\frac{\partial V_m^j}{\partial z_m}\right|(1+\Delta_{\theta_j}\phi_j). $$
Moreover, as $u_j$ is stationary at $p$ we get that 
$$ V_j(u_j) = C_1V_j(\phi_j)u_j-e^{-C_1\phi_j}V_j(\Delta_{\theta_j}\phi_j) $$
vanishes at $p$, hence
$$ \left.\left(e^{-C_1\phi_j}V_j(\Delta_{\theta_j}\phi_j)\right)\right|_p = \left(C_1V_j(\phi_j)u_j\right)|_p. $$
We conclude that 
$$ e^{-C_1\phi_j}\Delta_{\theta_j}(g_j+V_j(\phi_j)) \leq \left(\sup_m\left|\frac{\partial V_m^j}{\partial z_m}\right|+C_1V_j(\phi_j)\right)u_j. $$
By Lemma~\ref{lemma:VFbound} this is bounded by $Cu_j$ for a uniform constant $C$. 

We will now plug this into \eqref{eq:theta_iEquation3}. By the maximum principle $\Delta_{\omega_j} u_j\leq 0$ at $p$. Letting $M_i=\max_X u_i\geq 0$ we get
$$ 0\geq -Cu_j-t\sum_i M_i + C_2u_j^{\frac{n}{n-1}}-C_3 $$ 
at $p$. Summing over $j$ and using Young's inequality $a\leq \epsilon a^{n/(n-1)} + C(n,\epsilon)$ we get, after adjusting $C_3$,
\begin{eqnarray} 
0 & \geq & -C\sum M_i -kt\sum M_i + \frac{C_2}{\epsilon}\sum M_i-C_3 \nonumber \\
& = & \left(-C-kt+\frac{C_2}{\epsilon}\right)\sum M_i -C_3. \nonumber
\end{eqnarray}
Choosing $\epsilon$ small enough that the expression in the parenthesis is positive gives an upper bound on $\sum M_j$. Since $M_i\geq 0$ for all $i$, this implies a bound on $\sup M_i = \sup |u_i|$. This proves the lemma.
\end{proof}


\begin{proof}[Proof of Theorem~\ref{thm:OpennessAndHigherOrderEstimates}. Second part: $C^{2,\alpha}-estimates$]
Here we will prove that the set of $t$ such that \eqref{eq:ContinuitySetup} is solvable is closed. 

By Lemma~\ref{lemma:LaplacianEstimate}, $|\Delta_{\theta_i} \phi_i|$ is bounded by a constant that depend only on $||\phi_i||_{C^0(X)}$ for all $i$. We wish to apply Theorem~1 in \cite{Wang}. To do this we need uniform bounds on the H\"older norms of $\phi_i$ and $V_i(\phi_i)$. These are implied by the uniform bounds on $\Delta_{\theta_i}\phi_i$. To see this, choose coordinates that are normal with respect to $\theta_i$ and such that $\theta_i+dd^c\phi_i$ is diagonal at a point $p$. Since 
$$
\theta_i+dd^c\phi_i = \sum \left(1+\frac{\partial^2\phi_i}{\partial z_m\partial \bar z_m}\right)dz_md\bar z_m >0 
$$
we get that $\frac{\partial^2\phi_i}{\partial z_m\partial \bar z_m}>-1$ for all $m$. Together with the bound
$$ \Delta_{\theta_i}\phi_i = \sum_m \frac{\partial^2\phi_i}{\partial z_m\partial \bar z_m} \leq C $$
this gives uniform bounds on $\left|\frac{\partial^2\phi_i}{\partial z_m\partial \bar z_l}\right|$ for all $m$ and $l$ and the bounds on the H\"older norms follow.

Combining this with the argument at the end of Section~\ref{sec:Openness}, we conclude that the set of $t$ such that \eqref{eq:ContinuitySetup} is solvable is non-empty, open and closed in $[0,1]$. It follows that \eqref{eq:ContinuitySetup} has a solution $(\phi_i)$ at $t=1$. Consequently, by Lemma~\ref{lemma:ComplexEquation} $(\theta_i+dd^c\phi_i)$ solves \eqref{eq:cKESoliton}.
\end{proof}


\section{$C^0$-estimates}
\label{sec:C0Estimates}
%
%
%
%
%
%
%
%
In this section $X$ will always be a toric Fano manifold. In other words $c_1(X)>0$ and, letting $n=\dim X$, there is an $n$-dimensional complex torus $(\C^*)^n$ acting on $X$ by bi-holomorphisms such that the action admits an open, dense and free orbit. The purpose of the section is to prove Theorem~\ref{thm:C0Estimates}. We will begin by recalling the well known correspondence between metrics on line bundles over toric varieties and convex functions in $\R^n$. 
As in the introduction we fix an action of $(\C^*)^n$ on $X$ and identify $(\C^*)^n$ with its open, dense and free orbit. Let $\theta$ be an $(S^1)^n$-invariant K\"ahler form on $X$ that arise as the curvature of a metric $|| \cdot ||$ on a toric line bundle over $X$. Let $P$ be the polytope associated to this toric line bundle. Assume $s_0$ is the $(\C^*)^n$-invariant section corresponding to the point $0\in P$. By the invariance $s_0$ is nonvanishing on $(\C^*)^n$ and the metric can be represented by a plurisubharmonic function $\psi$ on $(\C^*)^n$ by
$$ \psi = -\log|| s_0 ||^2. $$
Then $\psi$ satisfies $dd^c\psi = \theta$. Using toric coordinates 
$$ (x_1,\ldots,x_n) = (\log |z_1|,\ldots,\log |z_n|)\in \R^n $$
$\psi$ defines a convex function on $\R^n$ 
$$ f(x_1,\ldots,x_n):=\psi(e^{x_1},\ldots,e^{x_n}) $$
which will have the property $\overline{\nabla f(\R^n)}=P$. Moreover, in logarithmic coordinates $\sigma_i = \log z_i$ we have
\begin{equation} 
\label{eq:MetricRepresentation}
\sum_{ij} \frac{\partial^2 f}{\partial x_i\partial x_j} d\sigma_id\bar\sigma_j = dd^c\psi = \theta. 
\end{equation}
Now, for a convex polytope $P$, let $E(P)$ be the space of smooth, strictly convex functions $f$ such that 
$$ \overline{\nabla f(\R^n)} = P. $$
Then it is well known (see for example Proposition 3.3, page 687 in \cite{BermanBerndtsson}) that \eqref{eq:MetricRepresentation} gives a one to one correspondence between the $(S^1)^n$ invariant elements in $[\theta]$ and $E(P)$. 

As noted in the introduction, the correspondence above extends trivially to any $\theta$ such that $[\theta]$ can be written as a linear combination with positive real coefficients of K\"ahler classes that arise as the curvature of toric line bundles. On the other hand, we have the following general principle which we record for the convenience of the reader:
\begin{lemma}
\label{lemma:RClasses}
Let $\alpha$ be a K\"ahler class on a Fano manifold $X$. Then there are some ample line bundles $L_1,\ldots,L_m$ over $X$ and positive real coefficients $\lambda_1,\ldots,\lambda_m$ such that 
\begin{equation}
\alpha = \sum_i \lambda_i c_1(L_i).
\label{eq:RBundle}
\end{equation}
\end{lemma}
\begin{proof}
First of all, any K\"ahler class $\alpha$ can be written as \eqref{eq:RBundle} where the line bundles $L_i$ are not necessarily ample and the constants $\lambda_i$ are not necessarily positive. To see this, recall that the map 
$$c_1:H^1(X,\mathcal O^*) \rightarrow H^2(X,\Z)$$
is part of the following exact sequence 
$$ H^1(X,\mathcal O^*) \xrightarrow{c_1} H^2(X,\Z) \rightarrow H^2(X,\mathcal O). $$
By the Kodaira Vanishing Theorem, since $X$ is Fano, 
$$H^2(X,\mathcal O) = H^2(X,K_X-K_X)=0.$$ 
It follows that $c_1$ is surjective, hence any element in $H_{DR}^2(X)\cong H^2(X,\R)$ can be written as a linear combination over $\R$ of elements in the image of $c_1$. Note that this means the set of rational classes, in other words the set of classes of the form $q c_1(L)$ for some rational number $q$ and some line bundle $L$, is dense in $H^{(1,1)}(X)$.

Now, the cone of K\"ahler classes $K$ is open in $H^{(1,1)}(X)$. This means we can take a set of rational classes $\eta_1,\ldots,\eta_j$ in $K$ that span $H^{(1,1)}(X)$ over $\R$. Moreover, these classes define an open subcone of $K$, 
$$ C = \left\{\sum_i\lambda_i\eta_i: \lambda_i\in \R_+\right\}.$$ 
For any $\alpha \in K$ we may take a rational class $\eta_0$ in the open set $(\alpha-C)\cap K$ which is nonempty since $\alpha$ is in the interior of $K$. This means $\alpha=\eta_0+\kappa$ where $\kappa\in C$ and \eqref{eq:RBundle} follows. 
\end{proof}

Noting that any divisor on a toric manifold is linearly equivalent to an $(S^1)^n$-invariant divisor, Lemma~\ref{lemma:RClasses} and the discussion preceding it gives:
\begin{lemma}
\label{lemma:ClassPolytope}
Let $\alpha$ be a K\"ahler class on $X$ and $P$ be the polytope corresponding to $\alpha$. Then \eqref{eq:MetricRepresentation} gives a one to one correspondence between the $(S^1)^n$ invariant elements in $\alpha$ and $E(P)$. Moreover, if $\alpha=c_1(L)$ where $L$ is a toric line bundle over $X$, then this correspondence is given by $\theta\mapsto f$ where
$$ f(\log|z_1|,\ldots,\log|z_n|) := -\log ||s_0||^2 $$
where $s_0$ is the $(S^1)^n$-invariant (meromorphic) section corresponding to the point $0\in \M\otimes \R$ and $||\cdot||$ is the metric on $L$ with curvature $\theta$. 
\end{lemma}

For each $i$, let $h_i:\R^n\rightarrow \R$ be defined by
$$ h_i(x) = \log\frac{1}{N_P}\sum_y e^{\langle y, x\rangle} $$
where the sum is taken over all vertices of the polytope $P_i$ and $N_P$ is the number of vertices of the polytope $P_i$. These functions are smooth, strictly convex and satisfy $\overline {\nabla h_i(\R^n)}=P_i$, hence $h_i\in E(P_i)$. For each $i$, let $\theta_i$ be the element in $\alpha_i$ corresponding to $h_i$. Then there is a one to one correspondence between $E(P_i)$ and the smooth $(S^1)^n$-invariant elements of $\Psh(X,\theta_i)$ given by
\begin{equation}
\label{eq:PshConvex}
f_i(x)-h_i(x) = \phi_i(e^{x}). 
\end{equation}
Moreover, $h_i(0) = 0$ for each $i$. This means the normalization \eqref{eq:Normalization} is equivalent to
    \begin{equation}
    \label{eq:RealNormalization}
    f_1(0)=\ldots=f_k(0).
    \end{equation}  
Using the correspondence in \eqref{eq:PshConvex}, it is possible to rewrite \eqref{eq:ContinuitySetup} to a real Monge-Amp\`ere equation. 

\begin{lemma}
    \label{lemma:RealMA}
    Assume $(\phi_i)$ and $(f_i)$ are related as in \eqref{eq:PshConvex}. Then, for $t\in [0,1]$, $(\phi)$ satisfies \eqref{eq:ContinuitySetup} if and only if $(f_i)$ satisfies
    \begin{eqnarray}
        \frac{e^{\langle V_1,\nabla f_1\rangle}}{\Vol_{V_1}(P_1)}\det\left(\frac{\partial^2 f_1}{\partial x_m\partial x_l}\right) = \ldots & = & \frac{e^{\langle V_k,\nabla f_k\rangle}}{\Vol_{V_k}(P_k)} \det\left(\frac{\partial^2 f_k}{\partial x_m\partial x_l}\right) \nonumber \\
        & = & e^{-t\sum_i f_i-(1-t)\sum_i h_i}. 
        \label{eq:RealContinuitySetup}
    \end{eqnarray}
\end{lemma}
\begin{proof}
First of all, using \eqref{eq:MetricRepresentation} we see that
\begin{eqnarray}
(\theta_i+dd^c\phi_i)^n & = & \left(\sum_{m,l}\frac{\partial^2 f_i}{\partial x_m\partial x_l}d\sigma_jd\bar\sigma_l\right)^n \nonumber \\
& = &  \det\left(\frac{\partial^2 f_i}{\partial x_m\partial x_l}\right)d\sigma d\bar\sigma, \label{eq:RealComplexMA}
\end{eqnarray}
where $d\sigma d\bar\sigma = d\sigma_1\ldots d\sigma_nd\bar\sigma_1\ldots d\bar\sigma_n$. 

Abusing notation, we may think of $f_i$ and $h_i$ as $(S^1)^n$-invariant plurisubharmonic functions on $(\C^*)^n\subset X$. We will show that 
\begin{equation}
    \label{eq:ToricRHS}
    e^{-t\sum_i\phi_i}\omega_0^n = e^{-t\sum_i (f_i-h_i)}\omega_0^n = e^{-t\sum_i f_i-(1-t)\sum_i h_i}d\sigma d\bar\sigma.
\end{equation}
This will follow if we show that 
\begin{equation}
\label{eq:VolumeForms}
e^{\sum h_i}\omega_0^n = d\sigma d\bar\sigma. 
\end{equation}
To do this, we note that by convexity 
$$ \overline{\nabla\left(\sum_i h_i\right)(\R^n)} = \overline{\sum_i \nabla h_i(\R^n)} = \sum P_i = P_{-K_X}. $$
By Lemma~\ref{lemma:ClassPolytope}, $\sum h_i$  defines a metric on $-K_X$ of curvature $\sum \theta_i$ by the relation
$$ ||s_0||_{\sum h_i}^2 = e^{-\sum h_i} $$
where $s_0$ is the unique $(\C^*)^n$-invariant section of $-K_X$, in other words
$$ s_0=\frac{\partial}{\partial \sigma_1}\wedge\ldots\wedge\frac{\partial}{\partial \sigma_k} =  d\sigma^{-1}. $$
Moreover, the volume form $\omega_0^n$ defines a metric on $-K_X$ by the relation
$$ ||d\sigma^{-1}||^2_{\omega_0^n} = \frac{\omega_0^n}{d\sigma d\bar\sigma}. $$
The curvature of $||\cdot||_{\omega_0^n}$ is $\Ric \omega_0 = \sum \theta_i$. By uniqueness in the Calabi-Yau Theorem $||\cdot||_{\sum h_k}=||\cdot||_{\omega_0^n}$ and \eqref{eq:VolumeForms} follows. 

It remains to show that 
\begin{equation} 
\label{eq:GradFactor} 
\frac{e^{\langle V_i,\nabla f_i \rangle}}{\Vol_{V_i}(P_i)} = e^{g_i+{V_i}(\phi_i)}.
\end{equation}
We will first show that 
\begin{equation} 
\label{eq:GradTerm} 
\langle V_i,\nabla f_i \rangle + C_i= g_i+{V_i}(\phi_i),
\end{equation} 
for some $C_i\in R$. Abusing notation again, and thinking of $f_i$ as an $(S^1)^n$-invariant plurisubharmonic function on $(\C^*)^n\subset X$, we compute
\begin{eqnarray}
dd^c\langle V_i,\nabla f_i \rangle & = & dd^c\left(\sum_m \frac{\partial f_i}{\partial x_m} a_m\right) \nonumber \\
& = & \sum_{m,j,l} \frac{\partial^3 f_i}{\partial x_j\partial x_l\partial x_m}a_m d\sigma_j d\bar\sigma_l \nonumber \\
& = & \partial i_V \left(\sum_{m,l} \frac{\partial^2 f_i}{\partial x_m\partial x_l} d\sigma_m d\bar\sigma_l\right) \nonumber \\
& = & \partial i_V (\theta_i+dd^c\phi_i) \nonumber \\
& = & L_V(\theta_i) \nonumber \\
& = & dd^c(g_i+{V_i}(\phi_i)) \nonumber
\end{eqnarray}
and \eqref{eq:GradTerm} follows by the maximum principle. To get \eqref{eq:GradFactor}, note that the push forward of $d\sigma d\bar\sigma$ under the map $(z_1,\ldots,z_n)\mapsto (\log|z_1|,\ldots,\log|z_n|)$ is the Euclidean measure $dx$ on $\R^n$. This means, by \eqref{eq:RealComplexMA} and \eqref{eq:GradTerm},
\begin{equation}
    \label{eq:GradFactorIntegral}
    \int_X e^{g_i+V_i(\phi_i)}\left(\theta_i+dd^c\phi_i\right)^n = \int_{\R^n} \det\left(\frac{\partial^2 f_i}{\partial x_m\partial x_l}\right) e^{\langle V_i,\nabla f_i\rangle+C_i}dx.
\end{equation}
Performing the change of variables $\nabla f_i = p$ we get
$$ \eqref{eq:GradFactorIntegral} = e^{C_i} \int_{P_i} e^{\langle V_i,p\rangle} dp.  $$
By \eqref{eq:WeightedVolumeInvariant} 
$$ \int_X e^{g_i+V_i(\phi_i)}\left(\theta_i+dd^c\phi_i\right)^n = \int_X e^{g_i}\theta_i^n = 1 $$
This means $C=\log\Vol_{V_i}(P_i)$ and \eqref{eq:GradFactor} follows. 

Using \eqref{eq:RealComplexMA}, \eqref{eq:ToricRHS} and \eqref{eq:GradFactor} we conclude that $(f_i)$ satisfies \eqref{eq:RealMA} if and only if $(\phi_i)$ satisfies \eqref{eq:ContinuitySetup} on $(\C^*)^n$. As $(\phi_i)$ is assumed to be smooth, the lemma follows. 
\end{proof}



\subsection{Estimates}
\label{sec:Estimates}
To prove Theorem~\ref{thm:C0Estimates} we need to prove that for all $t_0>0$ there is a constant $C$ such that any solution $(f_i)$ to \eqref{eq:RealContinuitySetup} at $t>t_0$, normalized according to \eqref{eq:RealNormalization},
satisfies
\begin{equation}
\sup_X | f_i - g_i | \leq C 
\label{eq:C0}
\end{equation}
for all $i$. 

For each $i$, let $u_i$ be the Legendre transform of $f_i$. Recall that $f_i$ is a smooth, strictly convex function on $\R^n$ such that $\overline{\nabla f_i(\R^n)} = P_i$. This means each $u_i$ is a smooth, strictly convex function on $P_i$. Moreover, a standard property of the Legendre transform is that
$$ \sup_{\R^n} |f_i - h_i| = \sup_{P_i} |u_i - h^*_i| $$
where $h_i^*$ is the Legendre transform of $h_i$. Since $h_i^*$ is bounded on $P_i$ (this is easy to verify) we have that \eqref{eq:C0} is equivalent to a uniform bound on $\sup_{P_i} |u_i|$. 

We will use a variant of the method of Wang and Zhu \cite{WangZhu} (see also \cite{Donaldson}). The first step is to establsih bounds on the function
$$ w = w_t = \sum_i \left(tf_i + (1-t)h_i\right). $$
Since $w$ is strictly convex and 0 is in the interior of $P_{-K_X}=\overline{\nabla w(\R^n)}$ we have that $w$ is bounded from below and attains its minimial value at a unique point. Let $m=\inf w$ and let $x_w$ be the minimal point of $w$.
\begin{lemma}\label{lemma:ControlOfW}
Assume $t_0>0$ and \eqref{eq:SolitonCondition} holds. Then there are constants $C$ and $\epsilon$ such that if $(f_i)$ is a solution to \eqref{eq:RealContinuitySetup} at $t>t_0$, then 
\begin{equation} 
\label{eq:wBound}
w\geq \epsilon|x-x_w|-C
\end{equation}
and 
\begin{equation}
\label{eq:x0Bound}
|x_w|\leq C. 
\end{equation}
\end{lemma}
The proof of Lemma~\ref{lemma:ControlOfW} follows one of the arguments in \cite{Donaldson} which is based on \cite{WangZhu}. The main point is the following convex geometric fact (see Proposition 2 in \cite{Donaldson})
\begin{lemma}
\label{lemma:KVolume}
Assume $f$ is a convex function on $\R^n$ attaining minimal value $0$, and suppose 
$$\det\left(\frac{\partial^2 f}{\partial x_m\partial x_l}\right)\geq \lambda$$
on $K = \{f \leq 1\}$. Then 
$$ \Vol(K)\leq C\lambda^{-1/2} $$
for some constant $C$ depending only on the dimension $n$. 
\end{lemma}

Using Lemma~\ref{lemma:KVolume} we can prove Lemma~\ref{lemma:ControlOfW}.
\begin{proof}[Proof of Lemma~\ref{lemma:ControlOfW}]
The proof proceeds in four steps:
\paragraph{Step 1: $m$ is bounded from below.} Let $\rho_{-K_X}$ be the support function of $P_{-K_X}$ defined by
$$ \rho_{-K_X}(x)=\sup_{p\in P_{-K_X}} \langle x,p \rangle. $$
Since $\nabla w(\R^n)=P_{-K_X}$ we have $w\leq m+\rho_{-K_X}$. Moreover, by the change of variables $p=\nabla f_i$
\begin{eqnarray} 
1 & = & \frac{\int_{P_i} e^{\langle V_i,p\rangle} dp}{\Vol_{V_i}(P_i)} \nonumber \\
& = & \int_{\R^n}  \frac{ e^{\langle V_i,\nabla f_i\rangle} }{\Vol_{V_i}(P_i)}
\det\left(\frac{\partial^2 f_i}{\partial x_m\partial x_l}\right)dx \nonumber \\
& = & \int_{\R^n}e^{-w}dx \nonumber \\
& \geq & Ce^{-m}\int_{\R^n} e^{-\rho_{-K_X}} dx \nonumber \\
& \geq & Ce^{-m}, \nonumber
\end{eqnarray}
possibly changing $C$ in the last inequality. This means $m$ is bounded from below by a uniform constant. 

\paragraph{Step 2: $m$ is bounded from above.}
By monotonicity of the determinant function and convexity
we have
\begin{eqnarray}
\det\left(\frac{\partial^2 w}{\partial x_m\partial x_l}\right)  & = & \det\left[t\sum_i\left(\frac{\partial^2 f_i}{\partial x_m\partial x_l}\right)+(1-t)\sum_i\left(\frac{\partial^2 h_i}{\partial x_m\partial x_l}\right)\right] \nonumber \\
& \geq & t_0^n\det\left(\frac{\partial^2 f_j}{\partial x_m\partial x_l}\right) \nonumber \\
& = & t_0^n \Vol_{V_j}(P_j)e^{-\langle V_j,\nabla f_j\rangle-w} \nonumber \\
& \geq & Ce^{-w}dx,
\nonumber
\end{eqnarray}
where the last inequality follows from the fact that $\overline{\nabla f_j (\R^n)} = P_j$ is bounded. This means $ \det\left(\frac{\partial^2 w}{\partial x_m\partial x_l}\right)\geq Ce^{-m-1}$ on $K = \{w\leq m+1\}$. By Lemma~\ref{lemma:KVolume}, possibly redefining $C$,
\begin{equation}
\label{eq:KVolume}
\Vol(K) \leq C e^{m/2}.
\end{equation}
Convexity of $w$ and the co-area formula gives
$$ 1=\int_{\R^n} e^{-w} dx \leq Ce^{-m/2}. $$
This means $m$ is bounded from above.

\paragraph{Step 3: $w\geq \epsilon |\cdot-x_w| - m + 1$ for uniform constants $\epsilon$ and $C$.} 
Since $\overline{\nabla w(\R^n)}=P_{-K_X}$ and $P_{-K_X}$ is bounded we have that there is a uniform constant $r>0$ such that $K$ contains a small ball centered at $x_w$ of radius $r$. If there was a point in $K$ far from $x_w$ then the volume of $K$ would be very big, contradicting \eqref{eq:KVolume}. This means $K$ is contained in a ball centered at $x_w$ of radius $R$ for some uniform constant $R$. Convexity of $w$ gives
$$ 
w(x) \geq 
\begin{cases}
R^{-1}|x-x_w|+m & \textnormal{if } x\notin K \\
m & \textnormal{if } x\in K
\end{cases}
$$
Moreover, $R^{-1}|x-x_w|\leq 1$ on $K$. This means putting $\epsilon = 1/R$ finishes Step 3.

\paragraph{Step 4: $|x_w|$ is bounded.} 
In this step we will use the assumption \eqref{eq:SolitonCondition}. By the Divergence Theorem, since $e^{-w}\rightarrow 0$ exponentially as $|x|\rightarrow \infty$,
$$ \int_{\R^n} \nabla  we^{-w}dx = \int_{\R^n} \div \nabla\left( e^{-w}\right)dx = 0. $$
Moreover, 
\begin{eqnarray}
\int_{\R^n} \nabla \left(\sum_i f_i\right) e^{-w} dx & = & \sum_i \int_{\R^n} \nabla f_ie^{-w}dx \nonumber \\
& = & \sum_i \frac{1}{\Vol_{V_i}(P_i)}\int_{\R^n} \nabla f_ie^{\langle V_i, \nabla f_i\rangle}  \det\left(\frac{\partial^2 f_i}{\partial x_m\partial x_l}\right)dx \nonumber \\
& = & \sum_i \frac{1}{\Vol_{V_i}(P_i)} \int_{P_i} pe^{\langle V_i, p\rangle} dp = 0, \nonumber
\end{eqnarray}
where the last two equalities are given by performing the change of variables $p=\nabla f_i(x)$ in each summand and \eqref{eq:SolitonCondition}.
This means 
\begin{equation}
\label{eq:fBarycenter}
\int_{\R^n} \nabla \left(\sum_i h_i\right) e^{-w} dx = 0.
\end{equation}
Recall that $\sum h_i$ is convex and hence $\nabla \left(\sum_i h_i\right)$ is monotone. Hence, if $|x_w|$ is large then, putting $v=x_w/|x_w|$, we get that $\langle x,v \rangle$ is positive and bounded away from 0 on some large ball centered at $x_w$. By \eqref{eq:wBound} the mass of $e^{-w}dx$ is concentrated around $x_w$. This contradicts \eqref{eq:fBarycenter}.
\end{proof}


We can now prove Theorem~\ref{thm:C0Estimates}.
\begin{proof}[Proof of Theorem~\ref{thm:C0Estimates}]
First of all, by the change of variables $x=\nabla u_i(p)$ and \eqref{eq:ContinuitySetup} we have
\begin{eqnarray} 
\int_{P_i} |\nabla u_i|^q dp & = & \int_{\R^n} |x|^q \det\left(\frac{\partial^2 f_i}{\partial x_m\partial x_l}\right) dx \nonumber \\
& \leq & \Vol_{V_i}(P_i)\int_{\R^n} |x|^q e^{-\langle V_i,\nabla f_i\rangle-w} dx \nonumber \\
& \leq & C\int_{\R^n} |x|^qe^{-w} dx \nonumber \\
& \leq & C_q \label{eq:LqNorm}
\end{eqnarray}
where the second inequality follows from the fact that $\overline{\nabla f_i(\R^n)}=P_i$ is bounded and the last inequality follows from Lemma~\ref{lemma:ControlOfW}. Put $q=n+1$ and 
$$ \hat u_i = \frac{1}{\Vol(P_i)}\int_{P_i} u_i dp. $$
By Morrey's inequality (see \cite{HaskovecSchmeiser}) we have
\begin{eqnarray}
|| u_i-\hat u_i ||_{C^{0,\gamma}(P_i)} & \leq & C || u_i-\hat u_i ||_{W^{1,q}(P_i)} \nonumber \\
& = & C || u_i-\hat u_i ||_{L^{q}(P_i)} + C || \nabla u_i ||_{L^{q}(P_i)}. 
\label{eq:Morreys}
\end{eqnarray}
where $\gamma = 1-n/q$. By the Poincar\'e-Wirtinger inequality this can be bounded by $$C|| \nabla u_i ||_{L^{q}(P_i)}$$ for some $C$. This is bounded by \eqref{eq:LqNorm}. Since $P_i$ is bounded we may conclude from this that 
\begin{equation} 
\sup_{p_1,p_2\in P_i} |u_i(p_1)-u_i(p_2)| \leq C|| u_i-\hat u_i ||_{C^{0,\gamma}(P_i)} \leq C.
\label{eq:uiVariation}
\end{equation}
This means it suffices to bound each $u_i$ in some point.

To bound each $u_i$ in some point, note that by general properties of Legendre transform $f_i(0)=-u_i(\nabla f_i(0))$. This means
$$ |u_i(\nabla f_i(0))| = |f_i(0)| = \frac{1}{k}\left|\sum_j f_j(0)\right| = \frac{1}{k}|w(0)|$$
where the last two equalities follow from \eqref{eq:RealNormalization} and the fact that $h_i(0)=0$ for all $i$. Since $|x_w|$ is bounded and $\nabla w \in P_{-K_X}$ is bounded we have that $|w(0)-w(x_w)|$ is bounded. By Lemma~\ref{lemma:ControlOfW}, $|w(x_w)|=|m|$ is bounded. This means $|u_i(\nabla f_i(0))|$ and hence, by \eqref{eq:uiVariation}, $\sup_{P_i}|u_i|$ is bounded for each $i$. By the discussion following \eqref{eq:C0} this proves the theorem.
%
\end{proof}


\begin{proof}[Proof of Theorem~\ref{thm:MainTheorem}.]
Assuming \eqref{eq:SolitonCondition} holds, existence of coupled K\"ahler-Ricci solitons follow directly from Theorem~\ref{thm:C0Estimates} and Theorem~\ref{thm:OpennessAndHigherOrderEstimates}. Indeed, any toric holomorphic vector field $V_i$ is in the reductive part of the Lie algebra of $\Aut(X)$. Moreover, $\Im V_i$ generates a compact one-parameter subgroup of $\Aut(X)$ and, since $\theta_i$ is $(S^1)^n$-invariant, $\Im L_V(\theta_i)=0$. 

Assume that $(\alpha_i)$ admits a coupled K\"ahler-Ricci soliton. By Lemma~\ref{lemma:ComplexEquation} and Lemma~\ref{lemma:RealMA}, \eqref{eq:RealMA} admits a solution. Then \eqref{eq:SolitonCondition} follows from Lemma~\ref{lemma:Obstruction} below.
\end{proof}

\begin{lemma}
    \label{lemma:Obstruction}
    Assume \eqref{eq:RealMA} admits a solution. Then 
    $$ \sum_i \mathcal A_{P_i}(V_i)=0. $$
\end{lemma}
\begin{proof}
Let $(f_i)$ be a solution to \eqref{eq:RealMA}. As in the proof of Lemma~\ref{lemma:ControlOfW}, by the Divergence Theorem, since $e^{-\sum f_i}\rightarrow 0$ exponentially as $|x|\rightarrow \infty$,
\begin{equation} 
\label{eq:Stokes}
\int_{\R^n} \left(\sum_i \nabla  f_i\right)e^{-\sum_i f_i}dx = \int_{\R^n} \div \nabla\left( e^{-\sum_i f_i}\right)dx = 0. 
\end{equation}
On the other hand, by \eqref{eq:RealMA}
$$ \eqref{eq:Stokes} = \sum_i \int_{\R^n} \nabla f_i e^{-\sum_i f_i}dx = \sum_i \int_{\R^n} \nabla f_i \frac{e^{\langle V_i,\nabla f_i\rangle }}{\Vol_{V_i}(P_i)}\det\left( \frac{\partial^2 f_i}{\partial x_m \partial x_l}\right)dx.$$
Performing the change of variables $\nabla f_i=p$ in each summand gives that the right hand side of this equals
$$ \sum_i \frac{1}{\Vol_{V_i}(P_i)}\int_{P_i} pe^{\langle V_i,p\rangle} dp = \sum_i \mathcal A_{P_i}(V_i). $$
This proves the lemma. 
\end{proof}
\begin{proof}[Proof of Corollary~\ref{cor:SpecialSoliton}.] 
Note that
\begin{equation}
    \label{eq:ASum}
    \sum_i \mathcal A_{P_i}(V)
\end{equation}
is the gradient of the function on $\R^n$ defined by
$$ V \mapsto \sum_i\log\int_{P_i} e^{\langle V,p\rangle}dp. $$
This is strictly convex and proper (in fact, its gradient image is $\sum_i P_i = P_{-K_X}$ which contain zero as an interior point), hence it admits a unique minimum. Letting $V$ be this minimum means \eqref{eq:SolitonCondition} is fulfilled. The corollary then follows from Theorem~\ref{thm:MainTheorem}.
\end{proof}
\begin{proof}[Proof of Corollary~\ref{cor:RealEquation}.]
The corollary follows from Theorem~\ref{thm:MainTheorem} and Lemma~\ref{lemma:RealMA}. 
\end{proof}

\subsection{Toric test configurations and proof of Theorem~\ref{thm:cYTD}}
\label{sec:TestConfigurations}
Theorem~\ref{thm:cYTD} will follow from Theorem~\ref{thm:MainTheorem} combined with Theorem~1.15 in \cite{HultgrenWittNystrom} and an explicit calculation of the Donaldson-Futaki invariant of test configurations induced by toric vector fields.

In \cite{HultgrenWittNystrom} a type of test configurations for decompositions of $c_1(X)$ was defined. The data defining them is essentially given by $k$ test configurations $(\mathcal X_1,\mathcal L_1), \ldots, (\mathcal X_k, \mathcal L_k)$ where $\mathcal X_1 = \ldots = \mathcal X_k =: \mathcal X$, such that $(\mathcal X,\sum_i\mathcal L_i)$ defines a test configuration for $(X,-K_X)$. The Donaldson-Futaki invariant associated to this data is defined as the intersection number
\begin{equation} 
    \label{eq:DF}
    DF(\mathcal X,(\mathcal L_i)) = -\sum_i \frac{\mathcal L^{n+1}}{|\alpha_i|} - (n+1)\frac{\left(-K_{\mathcal X/\mathbb P^1}-\sum_i \mathcal L_i\right)\cdot \left(\sum_i\mathcal L_i\right)^n}{(-K_X)^n} 
\end{equation}
where $|\alpha_i| = \int_X \theta^n$ for any $\theta$ such that $[\theta]=\alpha$. We point out that the notation here differs from \cite{HultgrenWittNystrom} in that here $(\mathcal X,\mathcal L_i)$ are the ($\C^*$-invariantly) compactified test configurations over $\mathbb P^1$. 

Now, recall that if $L$ is a toric line bundle over a toric manifold $X$, then a toric vector field $V$ induces a test configuration $(\mathcal X^V,\mathcal L^V)$ for $(X,L)$. This can be described in the following way: Let $d_1,\ldots,d_k\in N\otimes \R$ and $c_1,\ldots,c_k\in \R$ be the data defining the polytope $P_L$, i.e.
$$P_L = \{\langle d_i,\cdot \rangle \geq -c_i\}. $$
Then, the polytope of $\mathcal L^V$ can be arranged to be
$$P_{{\mathcal L^V}} = \{\langle d_i,\cdot\rangle \geq -c_i\} \cap \{\langle d_0+V,\cdot\rangle \geq 0 \} \cap \{\langle -d_0,\cdot\rangle \geq -C_{\mathcal L^V}\}. $$
where $d_0$ corresponds to the divisor given by the central fiber of $\mathcal X$ and $C_{\mathcal L^V}$ is a number that can be modified without changing the Donaldson-Futaki invariant by adding a multiple $\mathcal O_{\mathbb P^1}(1)$ to $\mathcal L^V$. 
In particular, as long as $C_{\mathcal L^V}$ is big enough for ${\mathcal L^V}$ to be ample, 
$$\left({\mathcal L^V}\right)^{n+1} = \Vol(P_{{\mathcal L^V}}) = \Vol(P_L)\left(C_{\mathcal L^V}+\left\langle V,b(P_L)\right\rangle\right). $$
This also gives
\begin{eqnarray}
(n+1)\mathcal O_{\mathbb P^1}(1)\cdot \left(\mathcal L^V\right)^n & = & \frac{d}{dt}\left(\mathcal L^V+tO_{\mathbb P^1}(1)\right)^{n+1} \nonumber \\
& = & \frac{d}{dt} \Vol\left(P_{\mathcal L^V+tO_{\mathbb P^1}(1)}\right) \nonumber \\
& = & \Vol(P_L). \end{eqnarray}
Finally, we note that if $L=-K_X$ then $\mathcal L^V$ is the relative canonical bundle of $\mathcal X^V$ up to a twist determined by $C_{\mathcal L^V}$.
\begin{equation}
    \label{eq:RelCan} 
    \mathcal L^V = -K_{\mathcal X^V/\mathbb P^1} +  C_{\mathcal L^V}\mathcal O_{\mathbb P^1}(1). 
\end{equation}

\begin{proof}[Proof of Theorem~\ref{thm:cYTD}]
Putting $V_1=\ldots=V_k=0$ gives
$$ \sum_i \mathcal A_{P_i}(V_i)=\sum_i b(P_i), $$
hence it follows from Theorem~\ref{thm:MainTheorem} that the third point of the theorem implies the first point. Moreover, the first point implies the second point by Theorem~1.15 in \cite{HultgrenWittNystrom}. Thus, to finish the proof of Theorem~\ref{thm:cYTD}, it suffices to prove that the second point implies the third point.  

We will prove the contrapositive. Assume $\sum_i b(P_i)\not= 0$, in other words $\sum_i \langle V,b(P_i)\rangle < 0$ for some toric vector field $V$. Let $(\mathcal X^V,(\mathcal L_i^V))$ be the associated test configuration. As $(\mathcal X^V,\sum_i \mathcal L_i^V)$ is a test configuration for $-K_X$ we get, using \eqref{eq:RelCan} and $|\alpha_i|=\Vol(P_i)$
\begin{eqnarray} 
DF\left(\mathcal X^V,\left(\mathcal L^V_i\right)\right) & = & \sum_i \frac{(\mathcal L_i^V)^{n+1}}{\Vol(P_i)} - (n+1)\frac{\left(\sum_i C_{\mathcal L_i^V}\right)\mathcal O_{\mathbb P^1}(1)\cdot \left(\sum_i\mathcal L_i^V\right)^n}{\Vol\left(P_{-K_X}\right)}. \nonumber \\
& = & \sum_i \left(C_{\mathcal L_i^V}+\left\langle V,b(P_L)\right\rangle\right) - \sum C_{\mathcal L_i^V} \nonumber \\
& = & \sum_i\left\langle V, b(P_i) \right\rangle < 0,
\end{eqnarray}
hence $(\alpha_i)$ is not K-polystable. 
\end{proof}

\subsection{Proof of Corollary~\ref{cor:example}}
\begin{proof}[Proof of Corollary~\ref{cor:example}]
First of all, by \cite{FMS} (see also \cite{Futaki} and \cite{Wang91}) the Futaki invariant of $X$ is non-zero, hence $X$ does not admit a Kähler-Einstein metrics. To prove the rest of the corollary, we fix a $(\mathbb C^*)^4$-action on $X$ in the following way: Consider the standard embeddings of $\mathcal O_{\mathbb P^2}(-1)$ and $\mathcal O_{\mathbb P^1}(-1)$ in to $\mathbb C^3\times \mathbb P^2$ and $\mathbb C^2\times \mathbb P^1$ respectively:
$$ \mathcal O_{\mathbb P^2}(-1) = \left\{ ((z_0,z_1,z_2),(a_0:a_1:a_2)) \, z_0a_1=z_1a_0, z_1a_2 = z_2a_1 \right\} $$
and 
$$ \mathcal O_{\mathbb P^1}(-1) = \left\{ ((w_0,w_1),(b_0:b_1)) \, w_0b_1=w_1b_0 \right\}. $$
We get an embedding of $X = \mathbb P(E)$ into $\mathbb P^4\times \mathbb P^2 \times \mathbb P^1$ as
\begin{eqnarray}
     X = \{ & ((z_0:z_1:z_2:w_0:w_1),(a_0:a_1:a_2),(b_0:b_1)): & \nonumber \\
    & z_0a_1=z_1a_0 & \nonumber \\
    & z_1a_2 = z_2a_1 & \nonumber \\
    & w_0b_1=w_1b_0 & \} \nonumber
\end{eqnarray}
We define a $(\mathbb C^*)^4$-action by letting an element $(t_1,t_2,t_3,t_4)\in (\mathbb C^*)^4$ act on $X$ by 
\begin{eqnarray}
& ((z_0:z_1:z_2:w_0:w_1),(a_0:a_1:a_2),(b_0:b_1)) & \nonumber \\
& \mapsto & \nonumber \\
& ((z_0:t_1z_1:t_2z_2:t_4w_0:t_4t_3w_1),(a_0:t_1a_1:t_2a_2),(b_0:t_3b_1)). & \nonumber
\end{eqnarray}
The invariant divisors are
\begin{eqnarray}
D_1 & = & \left\{ z_0 = a_0 = 0 \right\} \nonumber \\
D_2 & = & \left\{ z_1 = a_1 = 0 \right\} \nonumber \\
D_3 & = & \left\{ z_2 = a_2 = 0 \right\} \nonumber \\
D_4 & = & \left\{ w_0 = b_0 = 0 \right\} \nonumber \\
D_5 & = & \left\{ w_1 = b_1 = 0 \right\} \nonumber \\
D_6 & = & \left\{ z_0 = z_1 = z_2 = 0 \right\} \nonumber \\
D_7 & = & \left\{ w_0 = w_1 = 0 \right\} \nonumber
\end{eqnarray}
corresponding to the following elements in the lattice $N\cong\mathbb Z^4$ of one parameter subgroups of $(\mathbb C^*)^4$:
\begin{eqnarray}
d_1 & = & (-1,-1,0,-1) \nonumber \\
d_2 & = & (1,0,0,0) \nonumber \\
d_3 & = & (0,1,0,0) \nonumber \\
d_4 & = & (0,0,-1,1) \nonumber \\
d_5 & = & (0,0,1,0) \nonumber \\
d_6 & = & (0,0,0,-1) \nonumber \\
d_7 & = & (0,0,0,1). \nonumber
\end{eqnarray}
The divisor corresponding to $-K_X$ is $\sum_{i=1}^7 D_i$. For $c\in (1/4,3/4)$, we will be interested in divisors on the form
$$ D(c) = c(D_4+D_5) + \sum_{i\not=4,5} D_i/2. $$
corresponding to polytopes
\begin{equation} 
\label{eq:Polytopec}
P(c) = \{y\in \mathbb R^4: \langle y,d_i \rangle \leq 1\, i\not= 4,5, \, \langle y,d_i \rangle \leq c\ i=4,5   \}. 
\end{equation}
Note that the two classes in \eqref{eq:ExampleDecomposition} are given by $D(c)$ and $D(1-c)$, for 
\begin{equation} 
\label{eq:Correctc} 
c=\frac{1}{2} + \frac{\sqrt{\frac{5}{7}}}{4}\in \left(\frac{1}{4},\frac{3}{4}\right).
\end{equation}
To prove the proposition we will verify the follwoing two facts:
\begin{itemize}
    \item As long as $\in (\frac{1}{4},\frac{3}{4})$, none of the conditions in \eqref{eq:Polytopec} is redundant. (By standard theory for toric varieties this implies $D(c)$ and $D(-c)$ are ample and hence $\beta_1$ and $\beta_2$ are K\"ahler.) 
    \item The quantity $$\frac{\int_{P(c)} y dy}{\int_{P(c)} dy}+\frac{\int_{P(1-c)} y dy}{\int_{P(1-c)} dy} = 0.$$
    for $c$ as in \eqref{eq:Correctc}
\end{itemize}
Note that both these conditions are invariant under linear transformations of $R^n$. Applying the following linear transformation to the generators $d_1,\ldots,d_7$
$$ A = 
\begin{bmatrix} 
1 & 0 & 0 & -2 \\
0 & 1 & 0 & -2 \\
0 & 0 & 1 & 3 \\
0 & 0 & 0 & 6 
\end{bmatrix}
$$
gives new generators
\begin{eqnarray}
d_1' & = & (-1,-1,0,-2) \nonumber \\
d_2' & = & (1,0,0,-2) \nonumber \\
d_3' & = & (0,1,0,-2) \nonumber \\
d_4' & = & (0,0,-1,3) \nonumber \\
d_5' & = & (0,0,1,3) \nonumber \\
d_6' & = & (0,0,0,6) \nonumber \\
d_7' & = & (0,0,0,-6). \nonumber
\end{eqnarray}
And a new polytope 
\begin{equation}
\label{eq:P}
    P'(c) = \{y\in \mathbb R^4: \langle y,d'_i \rangle \leq 1\, i\not= 4,5, \, \langle y,d'_i \rangle \leq c\ i=4,5   \}
\end{equation} 
It is straight forward to check that as long as $c\in (1/4,3/4)$, non of the conditions in \eqref{eq:P} is redundant, hence $D(c)$ is ample for any $c\in (1/4,3/4)$. Moreover, the sets $\{d_1',d_2',d_3',d_6',d_7'\}$ and $\{d_4',d_5'\}$ are both invariant under the linear transformation 
$$ B = 
\begin{bmatrix} 
0 & -1 & 0 & 0 \\
1 & -1 & 0 & 0 \\
0 & 0 & -1 & 0 \\
0 & 0 & 0 & 1 
\end{bmatrix}.
$$
It follows that $P'(c)$ and hence the barycenter of $P'(c)$ is invariant under $B$. As any fixpoint of $B$ is paralell to $(0,0,0,1)$ we conclude that
$$ \int_{P'(c)} y_1 dy = \int_{P'(c)} y_2 dy = \int_{P'(c)} y_3 dy = 0.$$
Moreover, we denote by $S_2$ the two-dimensional simplex corresponding to the anti-canonical bundle of $\mathbb P^2$
$$ S_2 = \{y\in \mathbb R^2: -y_1\leq 1, -y_2\leq 1, y_1+y_2\leq 1\}$$
and note that $(y_1,\ldots,y_4)\in P'(c)$ if and only if $y_4\in (-1/12,1/12)$, $|y_3|\leq c-3y_4$ and $(y_1,y_2)\in (1/2+2y_4)S_2$. We get
\begin{eqnarray} \int_{P'(c)} y_4 dy & = & \int_{\frac{1}{12}[-1,1]} y_4 \left(\int_{(\frac{1}{2}+2y_4)S_2}dy_1dy_2\right)\left(\int_{(c-3y_4)[-1,1]}dy_3\right)dy_4 \nonumber \\
& = & 2\Vol(S_2)\int_{\frac{1}{12}[-1,1]} y_4\left(\frac{1}{2}+2y_4\right)^2(c-3y_4)dy_4  \nonumber \\
& = & \frac{5c-2}{720} \nonumber
\end{eqnarray}
and similarly
\begin{eqnarray} \int_{P'(c)} dy & = & 2\Vol(S_2)\int_{\frac{1}{12}[-1,1]} \left(\frac{1}{2}+2y_4\right)^2(c-3y_4)dy_4 \nonumber \\
& = & \frac{56c-3}{144}. \nonumber
\end{eqnarray}
It follows that
\begin{eqnarray}
\frac{\int_{P'(c)} y_4 dy}{\int_{P'(c)} dy} + \frac{\int_{P'(1-c)} y_4 dy}{\int_{P'(1-c)} dy}  & = & \frac{1}{5}\left(\frac{5c-2}{56c-3} + \frac{5(1-c)-2}{56(1-c)-3}\right) \nonumber \\
& = & \frac{(112c^2-112c+23)}{(56c-53)(56c-3)},
\end{eqnarray} 
which vanishes as 
$$c=\frac{1}{2}\pm \frac{\sqrt{\frac{5}{7}}}{4}\in \left(\frac{1}{4},\frac{3}{4}\right).$$
\end{proof}


\begin{thebibliography}{99}

\bibitem[BB13]{BermanBerndtsson} R. Berman and B. Berndtsson. \emph{Real Monge-Amp\`ere equations and K\"ahler-Ricci solitons on toric log Fano varieties.} Annales de la facult\'e des sciences de Toulouse (0240-2963). Vol. 22, 4 (2013): 649-711.  

\bibitem[Cao97]{Cao} H.D. Cao. \emph{Limits of solutions to the K\"ahler-Ricci flow.} J. Differential Geom., 45 (1997), 257-272.

\bibitem[CDS15]{ChenDonaldsonSun} X. Chen, S. Donaldson and S. Sun \emph{K\"ahler-Einstein  metrics  on  Fano  manifolds,  I-III.} J. Amer. Math. Soc. 28 (2015), 183-197, 199-234, 235-278.

\bibitem[CH12]{ChenHe} X. Chen and W. He. \emph{The complex Monge-Amp\`ere equation on compact K\"ahler manifolds.} Math. Ann. 354 (2012): 1583--1600.

\bibitem[Del17]{Delcroix} T. Delcroix. \emph{K\"ahler–Einstein metrics on group compactifications} Geometric and Functional Analysis (2017), vol 27, Issue 1, pp 78–129.

\bibitem[Don02]{Donaldson} S. Donaldson. \emph{Scalar curvature and stability of toric varities.} J. Differ. Geom. (2002) 62, pp. 289--349 

\bibitem[Don09]{Donaldson09} S. Donaldson \emph{Constant scalar curvature metrics on toric surfaces} Geom. funct. anal.
Vol. 19 (2009) pp. 83--136

\bibitem[Fut83]{Futaki} A. Futaki. \emph{An obstruction to the existence of Einstein Kähler metrics.} Invent. Math. 73 (1983), 437–443.

\bibitem[FMS90]{FMS} A. Futaki, T. Mabuchi and Y. Sakane. \emph{Einstein–Kähler metrics with positive Ricci curvature} Kähler metrics and moduli spaces, Adv. Stud. Pure Math., Vol. 18-II, Academic Press, 1990, pp. 11–83.

\bibitem[Ham93]{Hamilton93} R. Hamilton. \emph{Eternal solutions to the Ricci flow.} J. Differential Geom., 38 (1993),
pp. 1--11. 

\bibitem[Ham95]{Hamilton95} R. Hamilton. \emph{The Formation of Singularities in the Ricci Flow} Surveys In Differential Geometry, (1995) Vol. 2.  

\bibitem[HS09]{HaskovecSchmeiser} J. Haskovec and C. Schmeiser. \emph{A note on the anisotropic generalizations of the Sobolev and Morrey embedding theorems.} Monatsh. Math. 158 (2009), pp. 71--79.

\bibitem[HWN18]{HultgrenWittNystrom} J. Hultgren and D. Witt Nystr\"om. \emph{Coupled K\"ahler-Einstein metrics.} Int Math Res Not, rnx298 (2018), https://doi.org/10.1093/imrn/rnx298

\bibitem[Nic17]{Nicolaescu} L. Nicolaescu \emph{Lectures on the Geometry of Manifolds.} World Scientific, (2017) 2nd Edition.

\bibitem[Pin18]{Pingali} V. P. Pingali \emph{Existence of coupled K\"ahler-Einstein metrics using the continuity method.}  International Journal of Mathematics. Vol. 29, No. 05, 1850041 (2018)

\bibitem[Tia97]{Tian} K\"ahler-Einstein metrics with positive scalar curvature. Invent. Math., 130 (1997),
pp. 1--37. 

\bibitem[Tia15]{Tian15} G. Tian. \emph{K-stability and K\"ahler-Einstein metrics.} Comm. Pure Appl. Math. 68 (2015), no. 7, pp. 1085--1156.

\bibitem[TZ00]{TianZhu} G. Tian and X. Zhu. \emph{Uniqueness of K\"ahler-Ricci solitons.} Acta Math., 184 (2000) pp. 271--305.

\bibitem[TZ07]{TianZhu07} G. Tian and X. Zhu. \emph{Convergence of K\"ahler-Ricci flow} J. Amer. Math. Soc. 20 (2007) pp. 675--699.

\bibitem[WZ04]{WangZhu} X.J. Wang and X. Zhu. \emph{K\"ahler-Ricci solitons on toric manifolds with positive first Chern class.} Adv. Math. 188 (2004): 87--103.

\bibitem[WAN91]{Wang91} A.N. Wang. \emph{A note to Futaki’s invariant.} Chinese J. Math. 19 (1991), 239–242.

\bibitem[Wan12]{Wang} Y. Wang. \emph{On the $C^{2,\alpha}$-regularity of the complex Monge-Amp\`ere equation.} Math. Res. Lett. 19, no. 4 (1012): 939--946.

\bibitem[Yau78]{Yau} S.T. Yau. \emph{On The Ricci Curvature of a Compact K\"ahler Manifold and the Complex Monge-Amp\`ere Equation, I.} Comm. Pure App. Math., 31.3 (1978): 339--411. 

\bibitem[Zhu00]{Zhu} X. Zhu. \emph{K\"ahler-Ricci Soliton Typed Equations on Compact Complex Manifolds with $C_1(M)>0$.} J. Geom. Anal., Vol. 10, Nr. 4 (2000)

\bibitem[ZZ08]{ZhouZhu} B. Zhou and X. Zhu. \emph{K-stability on toric manifolds.} Proc. Amer. Math. Soc. 136, no. 9, pp. 3301--3307 (2008)

\end{thebibliography}
\end{document}